\documentclass[11pt,a4paper]{article}
\usepackage[utf8]{inputenc}
\usepackage[english]{babel}
\usepackage{amsmath}
\usepackage{amsfonts}
\usepackage{amssymb}
\usepackage{graphicx}
\usepackage{amsthm}
\usepackage{needspace}
\usepackage[margin=0.35in,format=hang,font=small,labelfont=bf,justification=centering,singlelinecheck=false]{caption}
\usepackage{enumerate}
\usepackage{mathtools}
\usepackage{multirow}
\usepackage{rotating}
\usepackage[section]{algorithm}
\usepackage{algpseudocode}
\usepackage[section]{placeins}
\usepackage{setspace}
\usepackage{tikz,pgfplots}
\usepackage[toc]{glossaries}
\pgfplotsset{compat=1.7}
\usetikzlibrary{decorations.markings, arrows, calc, shapes, hobby}

\usepackage{romannum}
\usepackage{longtable}

\usepackage{xcolor}

\usepackage{longtable}

\newtheorem{prop}{Proposition}
\newtheorem{lem}{Lemma}

\newtheorem{thm}{Theorem}

\newtheorem{constr}{Construction}

\usepackage{lscape}

\def\centrearc[#1](#2)(#3:#4:#5)
{ \draw[#1] ($(#2)+({#5*cos(#3)},{#5*sin(#3)})$) arc (#3:#4:#5); }

\author{Olivia Reade \\
The Open University}
\title{Chiral maps of given hyperbolic type on $A_k$}

\begin{document}

\maketitle

\begin{abstract}
This paper proves the existence of a chiral map with alternating automorphism group for every hyperbolic type. We present a set of constructions using permutations for when at least one parameter is even, and call on previously known results for when both the valency and the face-length are odd.
\end{abstract}

\section{Introduction and context}

It is known that chiral maps exist for any given hyperbolic type $\{m ,n \}$. In this paper we use permutation groups to prove the following theorem, the motivation being that we want a simple automorphism group.

\begin{thm}\label{thm:Thispaper}
Given a hyperbolic type $\{m ,n \}$, there exists a chiral map of that type with alternating automorphism group $A_k$, for some degree $k$.
\end{thm}

For context, the following theorem was published by Conder, Huc\'{i}kov\'{a}, Nedela and \v{S}ir\'{a}\v{n} \cite{CHNS}. The authors constructed base examples using permutation groups (yielding the theorem below) which are then used to prove that, for a given hyperbolic type, there are infinitely many chiral maps of that type with simple underlying graph.

\begin{thm}(Conder, Huc\'{i}kov\'{a}, Nedela, \v{S}ir\'{a}\v{n})\label{thm:CHNS} 
There exists an orientably-regular but chiral map of type $\{m,n \}$ with automorphism group $A_k$ or $S_k$ for some $k$ for every hyperbolic pair $(n,m)$.
\end{thm}

Asymptotically a quarter of the cases of Theorem \ref{thm:Thispaper} are therefore already proved by Theorem \ref{thm:CHNS}: when both $m$ and $n$ are odd the statement of Theorem \ref{thm:Thispaper} is an immediate corollary.

We also note that we do not need to address the case when $n=3$ since the following much stronger result has been known for some years thanks to Bujalance, Conder and Costa \cite{BujConCos} who proved the following.

\begin{thm}{(Bujalance, Conder, Costa)}\label{thm:BCC}
For all but finitely many $k$, for each $m\geq 7$ there exists a chiral map $\mathcal{M}$ of type $\{m, 3 \}$ with $Aut(\mathcal{M}) = A_k$.
\end{thm}

Another powerful theorem regarding chiral hypermaps is due to Jones whose proof uses spaces of differentials on Riemann surfaces and homology groups \cite{JonCC}. A corollary is as follows.

\begin{thm}{(Jones)}\label{thm:Joneschiralcovers}
Every reflexible map of arbitrary hyperbolic type is smoothly covered by infinitely many chiral maps.
\end{thm}

This statement also has a more elementary algebraic and map-theoretic proof presented in \cite{ReaSirCC} by Reade and \v{S}ir\'{a}\v{n}, and had it been known at the time, Theorem \ref{thm:Thispaper} would have provided a great shortcut for this alternative proof of Theorem \ref{thm:Joneschiralcovers}. We hope Theorem \ref{thm:Thispaper} may be useful in further situations where having a simple group makes things simpler.

The structure of the paper is as follows. In section \ref{sec:background} we recall some background information about chiral maps and permutation groups on which we will rely later. 
Since the dual of a chiral map is also chiral, we may work up to duality, allowing us to assume thenceforth that $m$ is even.
In sections \ref{sec:nodd} and \ref{sec:neven}, where $n$ is respectively odd or even, we present the constructions which then cover all but a small finite number of cases. For ease of reference, in each case we prove that the resulting group generated is alternating, and the associated map is chiral of the expected type. In section \ref{sec:proof} we prove the main theorem by calling on the work in the previous sections and addressing the missing cases.

\section{Background}\label{sec:background}

\subsection{Chiral maps}

A map is a cellular embedding of a connected graph on a surface such that the complement of the image of the graph is a disjoint union of regions each of which is homeomorphic to an open disc. Each directed edge of a map is known as an \emph{arc}. This paper concerns  orientably-regular maps - highly symmetric maps on orientable surfaces - which may be identified with their orientation-preserving automorphism groups.

For a given orientation of the surface and a given arc, the automorphisms $r$ and $s$ are defined to be the automorphisms of the map which act respectively and locally as the natural `one-step' rotation about the corresponding face and vertex.
The orientation-preserving automorphism group $G= \langle r, s \rangle$ of an orientably-regular map acts transitively on the arcs of the map and has partial presentation $G = \langle r,s \mid r^m, s^n, (rs)^2, \dots \rangle$. Such a map $\mathcal{M}$ is denoted $\mathcal{M}(G:r,s)$ and, using Sch\"{a}fli notation, has type $\{ m, n \}$.
An orientably-regular map $\mathcal{M}(G:r,s)$ is \emph{reflexible} if and only if there is an automorphism of the group $G = \langle r, s \rangle$ which inverts the generators. An orientably-regular map which is not reflexible is said to be \emph{chiral}. The study of symmetric maps on surfaces, including chiral maps, is well-established and further information can be found in \cite{JonSing, S}.

For the purposes of this paper it will be natural for us to work with a generating pair of elements for the group $G$ such that one of the generators is the involution $t := rs$, and the other is either $r$ or $s$. In particular the map is reflexible (not chiral) if and only if there is an automorphism of $G$ such that $t$ is fixed (which is equivalent to inverting $t$ since it is self-inverse) while the generator $r$ (or, equivalently, $s$) is inverted. As mentioned before, we will capitalise on the related fact that the dual of a chiral map of type $\{ m, n \}$ is chiral of type $\{n, m \}$.

\subsection{Permutation groups and primitivity}

Throughout this paper we define the group $G$ using permutations of some degree $k$, and illustrate the definitions using diagrams as follows.
Each of our permutation diagrams consists of $k$ points, is connected and defines the group $G = \langle r,t \rangle = \langle s,t \rangle$. Being an involution, $t$ is self-inverse and is the product of a set of disjoint transpositions, so $t$ is shown in the diagrams as a matching using blue edges. The other generator, be it $r$ or $s$, is defined according to its order and is illustrated using red cycles and red edges. For aesthetics, all loops to indicate a fixed point are omitted.

This is a method which has been used in the past to prove the existence of certain types of genuinely biregular maps, namely chiral rotary maps of type $\{ m, n \}$ in \cite{CHNS}, and truly edge-biregular maps of arbitrary feasible type in \cite{Thesis}, whose underlying automorphism group is either alternating or symmetric. As the authors in those works did, we make use of Jones' generalisation of Jordan's theorem, whose proof relies on the classification of finite simple groups, specifically the stated Corollary 1.3 in \cite{Jones}. For future reference, we state the relevant part of this powerful theorem:

\begin{thm}{(Jones)}\label{thm:JonesAltSym}
Let $G$ be a primitive permutation group of finite degree $k$, containing a cycle with $f$ fixed points. Then $G \geq \text{A}_k$ if $f\geq 3$.
\end{thm}

\smallskip

For some of the propositions in this paper, namely Propositions \ref{prop:noddmis4}, \ref{prop:noddmis6}, \ref{prop:noddmis8}, \ref{prop:nevenmis4}, \ref{prop:nevenmis6} and \ref{prop:nevenmis8}, we will be able to avoid relying on the classification of finite simple groups by using Jordan's original theorem:

\begin{thm}{(Jordan)}\label{thm:JordanAltSym}
Let $G$ be a primitive permutation group of finite degree $k$, containing a cycle of prime length which fixes at least three points. Then $G \geq \text{A}_k$.
\end{thm}

\subsection{Recognising chirality in permutation diagrams}

Remember that the map $\mathcal{M}(G:r,s)$ is reflexible (not chiral) if and only if there is an automorphism of the group $G$ such that the generators are inverted. In particular if the operation which inverts $s$ and fixes $t$ is not an automorphism of the group $G$ then the map is chiral.

Wild claim: The constructions we use are such that one can, by already knowing the specific structure of the group $G = A_k$, see by eye, from the lack of symmetry in the corresponding diagram, that the resulting orientably regular map is chiral. To formalise we use the following lemmas.

Each of the constructions in this paper has the group $G$ being the alternating group $A_k$ of the same degree as the diagram. In particular each automorphism $\phi$ of the group $G$ is equivalent to conjugation by an element of the symmetric group $S_k$ of that same degree, that is by a specific relabelling $\pi_\phi$ of the points on the permutation diagram.

\begin{lem}\label{lem:chiraldiag}
Let $G$ be such that $\mathcal{D}$ is a faithful permutation representation of $G = \langle s,t \mid s^n, t^2, (st)^m , \dots \rangle$ which is defined on $k$ points. Further suppose that $Aut(G) \leq S_k$. Then if the following circumstances are satisfied, this implies that the map $\mathcal (G:r,s)$ is chiral.
\begin{list}{•}{•}
\item There is a unique point $\zeta$ in $\mathcal{D}$ which is fixed by $s^bt$ for some non-zero $b$, such that $\zeta$ is not fixed by $t$. Note that such points are easy to spot in permutation diagrams.
\item There is an integer $c$ such that $\zeta s^c$ is fixed by $t$ and $\zeta t s^{-c}$ is not.
\end{list}
\end{lem}

These circumstances highlight the lack of reflective symmetry in the graph of the permutation diagram. Reflexibility of the underlying map, determined by a hypothetical orientation-reversing automorphism $\phi \in Aut(G)$ would demand that the two points $\zeta$ and $\zeta t$ were swapped by the relabelling $\pi \in S_k$ (induced by $\phi$) of the underlying permutation set. Furthermore it would also be the case that $\zeta s^c$ and $\zeta t s^{-c}$ were swapped by the same relabelling, but this turns out to be impossible by the second condition.

\begin{proof}
Let the conditions in the lemma be satisfied, and define $ \eta :=\zeta t$, noting also that $\eta = \zeta s^b$ is distinct from $\zeta$.

The map $\mathcal (G:r,s)$ is reflexible if and only if there is an involutory automorphism $\phi \in Aut(G)$ such that $\phi : s \to s^{-1}$ and $\phi : t \to t$. Suppose (for a contradiction) that this is the case, and that the corresponding relabelling of the points is $\pi \in S_k$ such that $s^\pi = s^{-1}$ and $t^\pi = t$.

Since there is the unique point $\zeta$ which is fixed by $s^b t$ (but not by $t$), and $\phi$ is an automorphism, there must be a unique point which is fixed by the image of $s^b t$ under $\phi$, that is which is fixed by $s^{-b} t$ (and not fixed by $ \phi(t) = t$). Notice that $\eta$ is, by its definition, a fixed point of $s^{-b}t$ which is not fixed by $t$ and hence $\pi$ interchanges the points labelled $\zeta$ and $\eta$.

Now $\pi$ is also necessarily such that $\zeta s^c$ is mapped to $ (\zeta \pi) \phi(s^c) = \eta s^{-c} = \zeta t s^{-c}$. But, by the conditions in the lemma, this means that $\pi$ maps a point which is fixed by $t$ to a point which is not fixed by $t$, which contradicts the definition of $\pi$ with respect to the $t$-fixing automorphism $\phi$.

We conclude that $(G : r,s)$ cannot be reflexible, and so it is chiral.
\end{proof}

Note that this proof relies heavily on the automorphism group of the group $G$ being contained within $S_k$. There are examples of non-symmetric permutation diagrams generated by $s$ and $t$ such that there is no orientation-reversing symmetry evident in the permutation diagram, but the resulting map is reflexible. An example is highlighted by Conder and Wilson in \cite{ConWil}. Using our notation this is equivalent to $s:=(123)(456)$ and $t:=(14)(67)$, and the difference here is that the resulting group $G = \langle s,t \rangle \cong PSL(2,7)$ which has automorphism group $PGL(2,7)$. Meanwhile, $PGL(2,7)$ itself, since it contains an element of order $8$, does not embed into $S_7$. There is no relabelling of the diagram which induces the orientation-reversing automorphism, and yet there is an element of $PGL(2,7)$ which inverts the corresponding generators $S$ and $T$ where $\langle S, T \rangle = PSL(2,7)$, so the map $(G ; r,s )$ is reflexible.

Mercifully, our aim is to construct examples whose automorphism group is $A_k$ which is known to have $S_k$ as its automorphism group (when $k \geq 7$), and as such this increase in minimum degree between the group $G$ and its automorphism group will not be a problem for us.

Some of our diagrams do not demonstrate the properties listed in the above lemma (even after substituting $r$ for $s$) and for most of those cases we may rely on the following lemma which is a very minor modification of the above, and whose proof is analogous.

\begin{lem}\label{lem:chiraldiagsfixedpoint}
Let $G$ be such that $\mathcal{D}$ is a transitive faithful permutation representation of $G = \langle s,t \mid s^n, t^2, (st)^m , \dots \rangle$ which is defined on $k$ points. Further suppose that $Aut(G) \leq S_k$. Then if the following circumstances are satisfied, this implies that the map $\mathcal (G:r,s)$ is chiral.
\begin{list}{•}{•}
\item There is a unique point $\zeta$ in $\mathcal{D}$ which is fixed by $s^bt$ for some non-zero $b$, such that $\zeta$ is not fixed by $t$.
\item There is an integer $c$ such that $\zeta s^c t$ is fixed by $s$ (or $s^2$) while $\zeta t s^{-c}t$ is not fixed by $s$ (respecitvely $s^2$).
\end{list}
\end{lem}

\begin{proof}
Omitted. Similar argument to that for the proof of Lemma \ref{lem:chiraldiag}.
\end{proof}

\section{When $m$ is even and $n$ is odd}\label{sec:nodd}

We have assumed that $m$ is even and our constructions depend on the relative sizes of $m$ and $n$. The notation is such that if there is a point labelled $b'$ it is the image under $t$ of the point labelled $b$. This notation is not ambiguous and so, for example, the image under $t$ of the point labelled $2a-1$ is denoted $2a-1'$ without brackets and without confusion.

\subsection{Even $m < n$ odd}

Since $n$ is odd, $s$ is necessarily an even permutation and so for small $m$ we build constructions based on $s$ being a single $n$-cycle, and then find $t$ such that it is the product of an even number of disjoint transpositions while at the same time ensuring their product, and so also $r$, has the correct order.

\begin{constr}\label{con:noddmis4} For $m=4$. See Figure \ref{fig:noddmis4}.

Let $n = 4a+4+i \geq 9$ where $i \in \{ -1, 1 \}$ and $a \geq 1$.

Define $t := (\alpha,\alpha ')(\beta, \beta') \prod_{j=1}^{2a}(j,j')$
and $s := s_i$ where 

\noindent $s_{-1} := (1,2, \dots , 2a-1,2a,\alpha,\alpha',\beta,{2a}',{2a-1}', \dots ,2',1')$ and

\noindent $s_{1} := (1,2, \dots , 2a-1,2a,\alpha,\beta,\gamma,\alpha',{2a}',{2a-1}', \dots ,2',1',\delta)$.

\end{constr} 

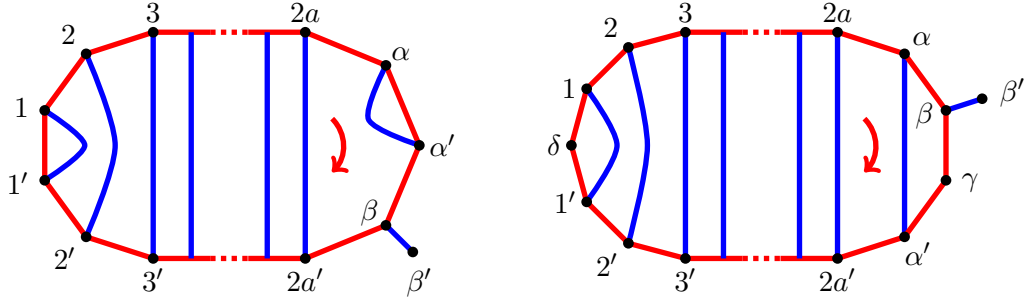
\begin{figure}

\begin{centering}
\begin{tikzpicture}[scale=0.5]

\foreach \numpt in {0,1,2,3,4}{\draw [red, line width=2pt]  (90+36*\numpt:3)--(90+36+36*\numpt:3);}

\draw [red, line width=2pt]  (0,3)--(1.5,3);
\draw [red, line width=2pt]  (0,-3)--(1.5,-3);
\draw [red, dotted, line width=2pt]  (1.5,3)--(2.5,3);
\draw [red, dotted, line width=2pt]  (1.5,-3)--(2.5,-3);
\draw [red, line width=2pt]  (2.5,3)--(4,3);
\draw [red, line width=2pt]  (2.5,-3)--(4,-3);

\draw[blue, line width=2pt] plot [smooth] coordinates {(162:3)  (180:1.8)  (-162:3)};

\draw[blue, line width=2pt] plot [smooth] coordinates {(126:3)  (180:1)  (-126:3)};

\draw [blue, line width=2pt]  (0,3)--(0,-3);
\draw [blue, line width=2pt]  (1,3)--(1,-3);

\foreach \numpt in {0,1,2,3,4,5}{\fill  (90+36*\numpt:3) circle(4pt);}

\draw [blue, line width=2pt]  (4,3)--(4,-3);
\draw [blue, line width=2pt]  (3,3)--(3,-3);

\node[above left] at (126:3) {$2$};
\node[below left] at (-126:3) {$2'$};

\node[left] at (-3,1) {$1$};
\node[left] at (-3,-1) {$1'$};

\node[above] at (0,3) {$3$};
\node[below] at (0,-3) {$3'$};

\draw [blue, line width=2pt, shift={(4,0)}]  (-45:3)--(-45:4);

\foreach \numpt in {-2,-1,0,1}{\draw [red, line width=2pt, shift={(4,0)}]  (45*\numpt:3)--(45+45*\numpt:3);}

\draw[blue, line width=2pt, shift={(4,0)}] plot [smooth] coordinates {(45:3)  (22.5:1.8)  (0:3)};

\centrearc[red, line width=2pt, ->](4,0)(45:-45:1)

\foreach \numpt in {-2,-1,0,1,2}{\fill[shift={(4,0)}]  (45*\numpt:3) circle(4pt);}
\fill[shift={(4,0)}]  (-45:4) circle(4pt);

\node[above] at (4,3) {$2a$};
\node[below] at (4,-3) {$2a'$};

\node[above right] at (6,2) {$\alpha$};
\node[right] at (7,0) {$\alpha '$};
\node[above left] at (6.2,-2.4) {$\beta $};
\node[below] at (7,-3) {$\beta '$};


\foreach \numpt in {-3,-2,-1,0,1,2}{\draw [red, line width=2pt, shift={(14,0)}]  (180+30*\numpt:3)--(180+30+30*\numpt:3);}

\draw [red, line width=2pt, shift={(14,0)}]  (0,3)--(1.5,3);
\draw [red, line width=2pt, shift={(14,0)}]  (0,-3)--(1.5,-3);
\draw [red, dotted, line width=2pt, shift={(14,0)}]  (1.5,3)--(2.5,3);
\draw [red, dotted, line width=2pt, shift={(14,0)}]  (1.5,-3)--(2.5,-3);
\draw [red, line width=2pt, shift={(14,0)}]  (2.5,3)--(4,3);
\draw [red, line width=2pt, shift={(14,0)}]  (2.5,-3)--(4,-3);

\draw[blue, line width=2pt, shift={(14,0)}] plot [smooth] coordinates {(150:3)  (180:1.8)  (-150:3)};

\draw[blue, line width=2pt, shift={(14,0)}] plot [smooth] coordinates {(120:3)  (180:1)  (-120:3)};

\draw [blue, line width=2pt, shift={(14,0)}]  (0,3)--(0,-3);
\draw [blue, line width=2pt, shift={(14,0)}]  (1,3)--(1,-3);

\foreach \numpt in {-3,-2,-1,0,1,2,3}{\fill[shift={(14,0)}] (180+30*\numpt:3) circle(4pt);}

\draw [blue, line width=2pt]  (18,3)--(18,-3);
\draw [blue, line width=2pt]  (17,3)--(17,-3);
\draw [blue, line width=2pt, shift={(18,0)}]  (54:3)--(-54:3);

\draw [blue, line width=2pt, shift={(18,0)}]  (18:3)--(18:4);
\fill[shift={(18,0)}] (18:4) circle(4pt);

\foreach \numpt in {0,1,2,3,4}{\draw [red, line width=2pt, shift={(18,0)}]  (90-36*\numpt:3)--(90-36-36*\numpt:3);}

\centrearc[red, line width=2pt, ->](18,0)(45:-45:1)

\foreach \numpt in {-3,-2,-1,0,1,2}{\fill[shift={(18,0)}]  (18+36*\numpt:3) circle(4pt);}

\node[above left, shift={(7,0)}] at (120:3) {$2$};
\node[below left, shift={(7,0)}] at (-120:3) {$2'$};

\node[left, shift={(7,0)}] at (150:3) {$1$};
\node[left, shift={(7,0)}] at (-150:3) {$1'$};

\node[left, shift={(7,0)}] at (-180:3) {$\delta$};

\node[above] at (14,3) {$3$};
\node[below] at (14,-3) {$3'$};

\node[above right] at (19.7,2.4) {$\alpha$};
\node[below right] at (19.5,-2.3) {$\alpha'$};

\node[right] at (22,1.4) {$\beta '$};
\node[right] at (21,-1) {$\gamma$};
\node[below left] at (20.8,1.4) {$\beta$};

\node[above] at (18,3) {$2a$};
\node[below] at (18,-3) {$2a'$};

\end{tikzpicture}
\caption{Construction \ref{con:noddmis4} when $i=-1$ and $i=1$ respectively}\label{fig:noddmis4}
\end{centering}
\end{figure}

\begin{prop}\label{prop:noddmis4}
Let $n$, $s$ and $t$ be defined according to Construction \ref{con:noddmis4}. Then $\mathcal{M}(G:r,s)$ is a chiral map of type $\{ 4,n \}$ and $G = A_{n+1}$.
\end{prop}

\begin{proof}
It is clear from their definitions that both $s$ and $t$ are even permutations and have the expected orders $n$ and $2$ respectively.
Now

\noindent $ s_{-1}t = (1,2')(2,3') \dots (2a-1,{2a}')(2a, \alpha', \beta', \beta)(1')(\alpha)$ while

\noindent $s_1t = (1,2')(2,3') \dots (2a-1,{2a}')(2a,\alpha')(\alpha,\beta',\beta,\gamma)(1',\delta)$ 
and both have order $4$.
The order of $r$ is the same as the order of $r^{-1} = st$, so $\mathcal{M}(G:r,s)$ is a map of type $\{ 4,n \}$.

In both cases $G = \langle s,t \rangle$ is certainly primitive since $s$ has just one fixed point, namely $\beta'$, and $s$ acts transitively on the set of all remaining points. When $i=-1$ the permutation $(s^3t)^2$ is a single $7$-cycle: in particular when $i=-1$, $a \geq 2$ and the permutation is $(\alpha,\beta',2a-2,2a-1,2a,\beta,\alpha')$ which fixes $4a-3$ points. Now consider the permutation $(s^2t)^2$ when $i=1$ which in every case is $(\alpha,2a,\beta,\gamma,\beta')$, a $5$-cycle fixing $4a+1$ points. Applying Jordan's Theorem \ref{thm:JordanAltSym} using the above permutations, cycles where there are enough fixed points, we see that $G$ contains $A_{n+1}$. Since in each case the group is generated by even permutations, $G$ must be the alternating group itself.

Then inspection of the corresponding diagrams in Figure \ref{fig:noddmis4} will yield that the map is chiral by, respectively, Lemma \ref{lem:chiraldiagsfixedpoint} for example using $b=4$ and $c=3$, and Lemma \ref{lem:chiraldiag} with $b=3$ and $c=2$.

\end{proof}

\begin{constr}\label{con:noddmis6} For $m=6$. See Figure \ref{fig:noddmis6}.

Let $n = 4a+6+i$ where $i \in \{ -1, 1 \}$ and $a \geq 1$.

Define $t := (\alpha,\alpha ')(\beta, \beta') t_i \prod_{j=1}^{2a}(j,j')$  where $t_{-1}$ is trivial and 

\noindent $t_{1} :=(\gamma,\gamma')(\delta,\delta')$,
and define $s := s_i$ where 

\noindent $s_{-1} := (1,2, \dots , 2a-1,2a,\alpha,\beta,\gamma,{2a}',{2a-1}', \dots ,2',1',\delta,\epsilon)$ and 

\noindent $s_{1} := (1,2, \dots , 2a-1,2a,\alpha,\beta,\beta',\gamma,\delta,\gamma',\alpha',{2a}',{2a-1}', \dots ,2',1')$.

\end{constr}
  
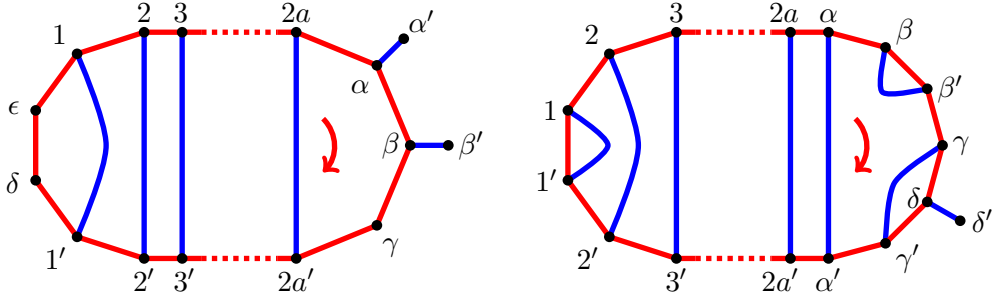
\begin{figure}

\begin{centering}
\begin{tikzpicture}[scale=0.5]

\foreach \numpt in {0,1,2,3,4}{\draw [red, line width=2pt]  (90+36*\numpt:3)--(90+36+36*\numpt:3);}

\draw [red, line width=2pt]  (0,3)--(1.5,3);
\draw [red, line width=2pt]  (0,-3)--(1.5,-3);
\draw [red, dotted, line width=2pt]  (1.5,3)--(3.5,3);
\draw [red, dotted, line width=2pt]  (1.5,-3)--(3.5,-3);
\draw [red, line width=2pt]  (3.5,3)--(4,3);
\draw [red, line width=2pt]  (3.5,-3)--(4,-3);

\draw[blue, line width=2pt] plot [smooth] coordinates {(126:3)  (180:1)  (-126:3)};

\draw [blue, line width=2pt]  (0,3)--(0,-3);
\draw [blue, line width=2pt]  (1,3)--(1,-3);
\fill (1,3) circle(4pt);
\fill (1,-3) circle(4pt);

\foreach \numpt in {0,1,2,3,4,5}{\fill  (90+36*\numpt:3) circle(4pt);}

\draw [blue, line width=2pt]  (4,3)--(4,-3);

\foreach \numpt in {-2,-1,0,1}{\draw [red, line width=2pt, shift={(4,0)}]  (45*\numpt:3)--(45+45*\numpt:3);}

\draw [blue, line width=2pt, shift={(4,0)}]  (45:3)--(45:4);
\fill[shift={(4,0)}]  (45:4) circle(4pt);

\draw [blue, line width=2pt]  (7,0)--(8,0);
\fill (8,0) circle(4pt);

\centrearc[red, line width=2pt, ->](4,0)(45:-45:1)

\foreach \numpt in {-2,-1,0,1,2}{\fill[shift={(4,0)}]  (45*\numpt:3) circle(4pt);}

\node[above left] at (126:3) {$1$};
\node[below left] at (-126:3) {$1'$};

\node[left] at (-3,1) {$\epsilon$};
\node[left] at (-3,-1) {$\delta$};

\node[above] at (0,3) {$2$};
\node[below] at (0,-3) {$2'$};

\node[above right] at (0.5,3) {$3$};
\node[below right] at (0.5,-3) {$3'$};

\node[above] at (4,3) {$2a$};
\node[below] at (4,-3) {$2a'$};

\node[below left] at (6.2,2) {$\alpha$};
\node[] at (7.3,3.3) {$\alpha '$};
\node[left] at (7,0) {$\beta $};
\node[right] at (8,0) {$\beta '$};
\node[below right] at (6,-2.2) {$\gamma$};


\foreach \numpt in {0,1,2,3,4}{\draw [red, line width=2pt, shift={(14,0)}]  (90+36*\numpt:3)--(90+36+36*\numpt:3);}

\draw [red, line width=2pt, shift={(14,0)}]  (0,3)--(0.5,3);
\draw [red, line width=2pt, shift={(14,0)}]  (0,-3)--(0.5,-3);
\draw [red, dotted, line width=2pt, shift={(14,0)}]  (0.5,3)--(2.5,3);
\draw [red, dotted, line width=2pt, shift={(14,0)}]  (0.5,-3)--(2.5,-3);
\draw [red, line width=2pt, shift={(14,0)}]  (2.5,3)--(4,3);
\draw [red, line width=2pt, shift={(14,0)}]  (2.5,-3)--(4,-3);

\draw[blue, line width=2pt, shift={(14,0)}] plot [smooth] coordinates {(162:3)  (180:1.8)  (-162:3)};

\draw[blue, line width=2pt, shift={(14,0)}] plot [smooth] coordinates {(126:3)  (180:1)  (-126:3)};

\draw [blue, line width=2pt, shift={(14,0)}]  (0,3)--(0,-3);

\foreach \numpt in {0,1,2,3,4,5}{\fill[shift={(14,0)}] (90+36*\numpt:3) circle(4pt);}

\draw [blue, line width=2pt]  (17,3)--(17,-3);

\draw[blue, line width=2pt, shift={(18,0)}] plot [smooth] coordinates {(60:3)  (45:2)  (30:3)};

\draw [blue, line width=2pt]  (18,3)--(18,-3);
\draw [blue, line width=2pt, shift={(18,0)}] plot [smooth] coordinates {(0:3) (-30:2) (-60:3)};

\foreach \numpt in {-3,-2,-1,0,1,2}{\draw [red, line width=2pt, shift={(18,0)}]  (30*\numpt:3)--(30+30*\numpt:3);}

\draw [blue, line width=2pt, shift={(18,0)}]  (-30:3)--(-30:4);

\centrearc[red, line width=2pt, ->](18,0)(45:-45:1)

\foreach \numpt in {-3,-2,-1,0,1,2,3}{\fill[shift={(18,0)}]  (30*\numpt:3) circle(4pt);}
\fill (17,3) circle(4pt);
\fill (17,-3) circle(4pt);
\fill [shift={(18,0)}] (-30:4) circle(4pt);

\node[above left, shift={(7,0)}] at (126:3) {$2$};
\node[below left, shift={(7,0)}] at (-126:3) {$2'$};

\node[left, shift={(7,0)}] at (162:3) {$1$};
\node[left, shift={(7,0)}] at (-162:3) {$1'$};

\node[above] at (14,3) {$3$};
\node[below] at (14,-3) {$3'$};

\node[above] at (18,3) {$\alpha$};
\node[below] at (18,-3) {$\alpha'$};

\node[above right] at (19.5,2.4) {$\beta $};
\node[right] at (20.6,1.5) {$\beta '$};
\node[right] at (21,0) {$\gamma$};

\node[left] at (20.7,-1.4) {$\delta $};
\node[right] at (21.5,-2) {$\delta '$};
\node[below, right] at (19.5,-3) {$\gamma '$};

\node[above left] at (17.5,3) {$2a$};
\node[below left] at (17.5,-3) {$2a'$};

\end{tikzpicture}
\caption{Construction \ref{con:noddmis6} when $i=-1$ and $i=1$ respectively}\label{fig:noddmis6}
\end{centering}
\end{figure}

\begin{prop}\label{prop:noddmis6}
Let $n$, $s$ and $t$ be defined as in Construction \ref{con:noddmis6}. Then $\mathcal{M}(G:r,s)$ is a chiral map of type $\{ 6,n \}$ and $G = A_{k}$.
\end{prop}

\begin{proof}
The definitions ensure both generators $s$ and $t$ are even permutations and have the required orders.
Now $st=r^{-1}$ and, when $i=-1$ 

\noindent $ st = (1,2')(2,3') \dots (2a-1,{2a}')(2a,\alpha',\alpha,\beta',\beta,\gamma)(1',\delta,\epsilon)$, and when $i=1$ 

\noindent $st = (1,2')(2,3') \dots (2a-1,{2a}')(2a,\alpha')(\alpha,\beta',\gamma')(\gamma,\delta',\delta)(1')(\beta)$, 
so in each case $r$ has order six.

In the case where $i=1$ the group $G = \langle s,t \rangle$ is certainly primitive since $s$ has just one fixed point, namely $\delta'$, and is transitive on all other points.
In the case where $i=-1$ the element $s$ is a single cycle fixing precisely $\alpha'$ and $\beta'$ while its conjugate $s^{tst}$ fixes $\alpha'$ and $2a'$. The stabiliser of $\alpha '$ is transitive on all other points and hence $G$ is primitive.

It remains to be proven is that $G$ is an alternating group in each case, and for this we seek a permutation which is a single cycle of prime length.
When $i=-1$ the permutation $(s^2t)^2 = (1', 2')(2a-1,\alpha)(2a,\beta,\beta')(\alpha',\gamma)(\delta,\epsilon)$, which when squared yields a single $3$-cycle.
In the other case, when $i=1$, the permutation $(s^2t)^2 = (2a,\delta',\alpha,\gamma',\beta',\delta,\beta)$, a $7$-cycle fixing $4a+1$ points.
This allows us to apply Jordan's Theorem \ref{thm:JordanAltSym} and in each case we see that $G$ contains the alternating group of the corresponding degree (respectively $k = n+2$ or $k = n+1$), and since it is generated by even permutations, $G$ must be the alternating group itself.

Inspection of the corresponding diagram in Figure \ref{fig:noddmis6}, respectively, using Lemma \ref{lem:chiraldiag} with $b=4$ and $c=3$ and using Lemma \ref{lem:chiraldiagsfixedpoint} with $b=6$ and $c=4$, yields that $\mathcal{M}(G:r,s)$ is a chiral map of type $\{ 6,n \}$.
\end{proof}

\begin{constr}\label{con:noddmis8} For $m=8$. See Figure \ref{con:noddmis8}.

Let $n = 4a+6+i$ where $i \in \{ 1, -1 \}$ and $a \geq 1$.

Define $t := (\alpha,\alpha ')(\beta, \beta')\prod_{j=1}^{2a}(j,j')$
and $s := s_i$ where 

\noindent $s_{-1} := (1,2, \dots , 2a-1,2a,\alpha,\beta,\gamma,\delta,\epsilon,{2a}',{2a-1}', \dots ,2',1')$ and

\noindent $s_{1} := (1,2, \dots , 2a-1,2a,\alpha,\alpha',\beta,\gamma,\delta,\epsilon,\zeta,{2a}',{2a-1}', \dots ,2',1')$.

\end{constr}
 
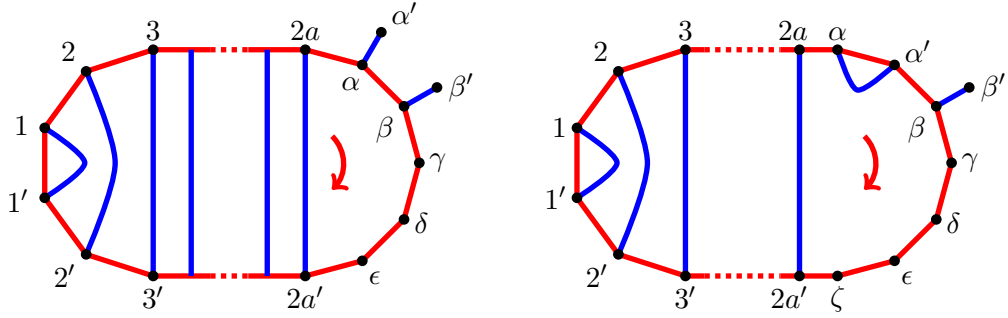
\begin{figure}

\begin{centering}
\begin{tikzpicture}[scale=0.5]

\foreach \numpt in {0,1,2,3,4}{\draw [red, line width=2pt]  (90+36*\numpt:3)--(90+36+36*\numpt:3);}

\draw [red, line width=2pt]  (0,3)--(1.5,3);
\draw [red, line width=2pt]  (0,-3)--(1.5,-3);
\draw [red, dotted, line width=2pt]  (1.5,3)--(2.5,3);
\draw [red, dotted, line width=2pt]  (1.5,-3)--(2.5,-3);
\draw [red, line width=2pt]  (2.5,3)--(4,3);
\draw [red, line width=2pt]  (2.5,-3)--(4,-3);

\draw[blue, line width=2pt] plot [smooth] coordinates {(162:3)  (180:1.8)  (-162:3)};

\draw[blue, line width=2pt] plot [smooth] coordinates {(126:3)  (180:1)  (-126:3)};

\draw [blue, line width=2pt]  (0,3)--(0,-3);
\draw [blue, line width=2pt]  (1,3)--(1,-3);

\foreach \numpt in {0,1,2,3,4,5}{\fill  (90+36*\numpt:3) circle(4pt);}

\draw [blue, line width=2pt]  (4,3)--(4,-3);
\draw [blue, line width=2pt]  (3,3)--(3,-3);

\node[above left] at (126:3) {$2$};
\node[below left] at (-126:3) {$2'$};

\node[left] at (-3,1) {$1$};
\node[left] at (-3,-1) {$1'$};

\node[above] at (0,3) {$3$};
\node[below] at (0,-3) {$3'$};

\foreach \numpt in {-3,-2,-1,0,1,2}{\draw [red, line width=2pt, shift={(4,0)}]  (30*\numpt:3)--(30+30*\numpt:3);}

\draw [blue, line width=2pt, shift={(4,0)}]  (60:3)--(60:4);
\fill[shift={(4,0)}] (60:4) circle(4pt);

\draw [blue, line width=2pt, shift={(4,0)}]  (30:3)--(30:4);
\fill[shift={(4,0)}] (30:4) circle(4pt);

\centrearc[red, line width=2pt, ->](4,0)(45:-45:1)

\foreach \numpt in {-3,-2,-1,0,1,2,3}{\fill[shift={(4,0)}]  (30*\numpt:3) circle(4pt);}
\node[] at (5.2,2.1) {$\alpha $};
\node[above right] at (6,3.4) {$\alpha '$};
\node[below left] at (6.6,1.5) {$\beta$};
\node[right] at (7.5,2) {$\beta '$};
\node[right] at (7,0) {$\gamma$};

\node[right] at (6.6,-1.6) {$\delta$};
\node[below right] at (5.4,-2.6) {$\epsilon$};

\node[above] at (4,3) {$2a$};
\node[below] at (4,-3) {$2a'$};


\foreach \numpt in {0,1,2,3,4}{\draw [red, line width=2pt, shift={(14,0)}]  (90+36*\numpt:3)--(90+36+36*\numpt:3);}

\draw [red, line width=2pt, shift={(14,0)}]  (0,3)--(0.5,3);
\draw [red, line width=2pt, shift={(14,0)}]  (0,-3)--(0.5,-3);
\draw [red, dotted, line width=2pt, shift={(14,0)}]  (0.5,3)--(2.5,3);
\draw [red, dotted, line width=2pt, shift={(14,0)}]  (0.5,-3)--(2.5,-3);
\draw [red, line width=2pt, shift={(14,0)}]  (2.5,3)--(4,3);
\draw [red, line width=2pt, shift={(14,0)}]  (2.5,-3)--(4,-3);

\draw[blue, line width=2pt, shift={(14,0)}] plot [smooth] coordinates {(162:3)  (180:1.8)  (-162:3)};

\draw[blue, line width=2pt, shift={(14,0)}] plot [smooth] coordinates {(126:3)  (180:1)  (-126:3)};

\draw [blue, line width=2pt, shift={(14,0)}]  (0,3)--(0,-3);

\foreach \numpt in {0,1,2,3,4,5}{\fill[shift={(14,0)}] (90+36*\numpt:3) circle(4pt);}

\draw [blue, line width=2pt]  (17,3)--(17,-3);

\draw[blue, line width=2pt, shift={(18,0)}] plot [smooth] coordinates {(60:3)  (75:2)  (90:3)};

\draw [blue, line width=2pt, shift={(18,0)}]  (30:3)--(30:4);
\fill[shift={(18,0)}] (30:4) circle(4pt);

\foreach \numpt in {-3,-2,-1,0,1,2}{\draw [red, line width=2pt, shift={(18,0)}]  (30*\numpt:3)--(30+30*\numpt:3);}

\centrearc[red, line width=2pt, ->](18,0)(45:-45:1)

\foreach \numpt in {-3,-2,-1,0,1,2,3}{\fill[shift={(18,0)}]  (30*\numpt:3) circle(4pt);}
\fill (17,3) circle(4pt);
\fill (17,-3) circle(4pt);

\node[above left, shift={(7,0)}] at (126:3) {$2$};
\node[below left, shift={(7,0)}] at (-126:3) {$2'$};

\node[left, shift={(7,0)}] at (162:3) {$1$};
\node[left, shift={(7,0)}] at (-162:3) {$1'$};

\node[above] at (14,3) {$3$};
\node[below] at (14,-3) {$3'$};

\node[above] at (18,3) {$\alpha$};
\node[below] at (18,-3) {$\zeta$};

\node[above right] at (19.5,2.4) {$\alpha '$};
\node[below left] at (20.6,1.5) {$\beta$};
\node[right] at (21.5,2) {$\beta '$};
\node[right] at (21,0) {$\gamma$};

\node[right] at (20.6,-1.6) {$\delta$};
\node[below right] at (19.4,-2.6) {$\epsilon$};

\node[above left] at (17.5,3) {$2a$};
\node[below left] at (17.5,-3) {$2a'$};

\end{tikzpicture}
\caption{Construction \ref{con:noddmis8} when $i=-1$ and $i=1$ respectively}\label{fig:noddmis8}
\end{centering}
\end{figure}

\begin{prop}\label{prop:noddmis8}
Let $n$, $s$ and $t$ be defined by Construction \ref{con:noddmis8}. Then the map $\mathcal{M}(G:r,s)$ is chiral of type $\{ 8,n \}$ and $G = A_k$.
\end{prop}

\begin{proof}(Sketch)
It is easy to check the generators are even and that $r,s$ and $t$ have the expected orders, so the corresponding map has type $\{ 8,n \}$. The group $G$ can be shown to be primitive by using similar arguments to those found in the proof of Proposition \ref{prop:noddmis6}. When $i=-1$ the element $(s^2t)^4$ is a $5$-cycle and when $i=1$ the element $(s^3t)^2$ is an $11$-cycle. Combined with Jordan's Theorem \ref{thm:JordanAltSym}, in all cases except when $a=1$ and $i=1$, we may conclude $G = \langle s,t \rangle$ is the alternating group $A_{k}$ as expected. The claim is also true for the map of type $\{ 8,11 \}$ where $G = A_{12}$.
In each case the group generated is $A_{k}$ and by Lemma \ref{lem:chiraldiag} with $b=8$ and $c=5$ with reference to Figure \ref{fig:noddmis8}, we may conclude that the corresponding map $\mathcal{M}(G:r,s)$ is chiral.
\end{proof}

When even $m \geq 10$ our constructions are built on $r$ being the product of an $m$-cycle and an odd number of disjoint transpositions, thereby ensuring that $r$ is an even permutation of order $m$.

\begin{constr}\label{con:noddmisgeq10} For $m \geq 10$. See Figure \ref{fig:noddmgeq10}.

Let $n = m + 4a + i$ where $i \in \{ 1, -1 \}$ and $a \geq 0$.

For $a \geq 1$ let $t_a := \Pi_{j=1}^{a}{(\alpha_j,\alpha_j')(\beta_j,\beta_j')}$ and $r_a := \Pi_{j=1}^{a}{(\alpha_j',\beta_j)(\beta_j',\alpha_{j+1})}$, and when $a=0$ we define $t_a$ and $r_a$ to be the identity. Let $t_i = (8,9)$ when $i=-1$, and when $i=1$ let $t_i = (\alpha_{a+1},\alpha_{a+1}')$. Define 

\noindent $r := (1,2, \dots ,m-1,m).(1',\alpha_1).r_a$ and 
$t := (1,1')(2,3)(4,5)(6,7).(m,m').t_a.t_i$.

\end{constr}
 
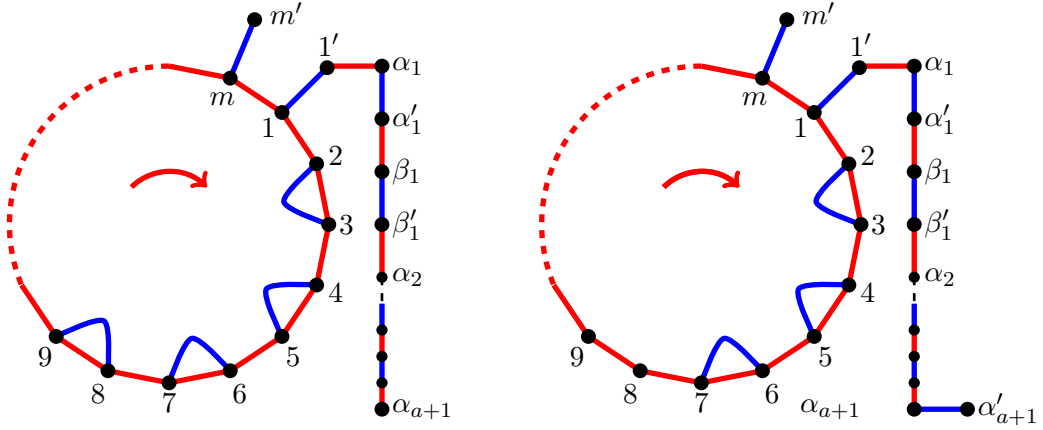
\begin{figure}

\begin{centering}
\begin{tikzpicture}[scale=0.7]

\foreach \numpt in {0,1,2,3,-1,-2,-3,-4,-5,-6,-7}{\draw[red, line width=2pt] (22.5*\numpt:3)--(22.5*\numpt+22.5:3);}

\node[above] at (3,3) {$1 '$};
\node[right] at (1.7,4) {$m '$};

\centrearc[red, line width=2pt, dashed](0,0)(90:202.5:3)

\draw[blue, line width=2pt] plot [smooth] coordinates {(0:3) (11.25:2.2) (22.5:3)};

\draw[blue, line width=2pt] plot [smooth] coordinates {(-22.5:3) (-33.75:2.2) (-45:3)};

\draw[blue, line width=2pt] plot [smooth] coordinates {(-67.5:3) (-78.75:2.2) (-90:3)};

\draw[blue, line width=2pt] plot [smooth] coordinates {(-112.5:3) (-123.75:2.2) (-135:3)};

\draw[blue, line width=2pt] (67.5:3)--(67.5:4.2);
\draw[blue, line width=2pt] (45:3)--(45:4.2);
\draw[red, line width=2pt] (3,3)--(4,3);
\node[right] at (4,3) {$\alpha_1$};
\node[right] at (4,2) {$\alpha_1 '$};
\node[right] at (4,1) {$\beta_1$};
\node[right] at (4,0) {$\beta_1 '$};
\draw[blue, line width=2pt] (4,3)--(4,2);
\draw[red, line width=2pt] (4,2)--(4,1);
\draw[blue, line width=2pt] (4,1)--(4,0);
\draw[red, line width=2pt] (4,0)--(4,-1);
\node[right] at (4,-1) {$\alpha_2$};
\draw[black, line width=1pt, dashed,] (4,-1)--(4,-1.5);
\draw[blue, line width=2pt] (4,-1.5)--(4,-2);
\draw[red, line width=2pt] (4,-2)--(4,-2.5);
\draw[blue, line width=2pt] (4,-2.5)--(4,-3);
\draw[red, line width=2pt] (4,-3)--(4,-3.5);
\node[right] at (4,-3.5) {$\alpha_{a+1}$};
\foreach \numpt in {-0.5,3,4,5,6}{\fill  (4,\numpt-3) circle(4pt);}

\foreach \numpt in {0,1,2,3,-1,-2,-3,-4,-5,-6}{\fill (22.5*\numpt:3) circle(4pt);}

\foreach \numpt in {-4,-5,-5.5,-6}{\fill (4,\numpt+3) circle(3pt);}

\foreach \numpt in {2,3}{\fill (22.5*\numpt:4.2) circle(4pt);}

\centrearc[red, line width=2pt, ->](0,0)(135:45:1)

\node[below] at (1,2.7) {$m$};
\node[below left] at (2.2,2.2) {$1$};
\node[right] at (2.8,1.3) {$2$};
\node[right] at (3,0) {$3$};
\node[right] at (2.8,-1.3) {$4$};
\node[right] at (2,-2.5) {$5$};
\node[right] at (1,-3.2) {$6$};
\node[below] at (0,-3) {$7$};
\node[left] at (-1,-3.2) {$8$};
\node[left] at (-2,-2.5) {$9$};

\centrearc[red, line width=2pt, ->](10,0)(135:45:1)

\foreach \numpt in {0,1,2,3,-1,-2,-3,-4,-5,-6,-7}{\draw[red, line width=2pt, shift={(10,0)}] (22.5*\numpt:3)--(22.5*\numpt+22.5:3);}

\node[above] at (13,3) {$1 '$};
\node[right] at (11.7,4) {$m '$};

\centrearc[red, line width=2pt, dashed](10,0)(90:202.5:3)

\draw[blue, line width=2pt, shift={(10,0)}] plot [smooth] coordinates {(0:3) (11.25:2.2) (22.5:3)};

\draw[blue, line width=2pt, shift={(10,0)}] plot [smooth] coordinates {(-22.5:3) (-33.75:2.2) (-45:3)};

\draw[blue, line width=2pt, shift={(10,0)}] plot [smooth] coordinates {(-67.5:3) (-78.75:2.2) (-90:3)};

\draw[blue, line width=2pt, shift={(10,0)}] (45:3)--(45:4.2);
\draw[red, line width=2pt, shift={(10,0)}] (3,3)--(4,3);
\draw[blue, line width=2pt, shift={(10,0)}] (67.5:3)--(67.5:4.2);
\node[right] at (14,3) {$\alpha_1$};
\draw[blue, line width=2pt, shift={(10,0)}] (4,3)--(4,2);
\node[right] at (14,2) {$\alpha_1 '$};
\draw[red, line width=2pt, shift={(10,0)}] (4,2)--(4,1);
\node[right] at (14,1) {$\beta_1$};
\draw[blue, line width=2pt, shift={(10,0)}] (4,1)--(4,0);
\node[right] at (14,0) {$\beta_1 '$};
\draw[red, line width=2pt, shift={(10,0)}] (4,0)--(4,-1);
\node[right] at (14,-1) {$\alpha_2$};
\draw[black, line width=1pt, dashed, shift={(10,0)}] (4,-1)--(4,-1.5);
\draw[blue, line width=2pt, shift={(10,0)}] (4,-1.5)--(4,-2);
\draw[red, line width=2pt, shift={(10,0)}] (4,-2)--(4,-2.5);
\draw[blue, line width=2pt, shift={(10,0)}] (4,-2.5)--(4,-3);
\draw[red, line width=2pt, shift={(10,0)}] (4,-3)--(4,-3.5);
\node[left] at (13.2,-3.5) {$\alpha_{a+1}$};
\draw[blue, line width=2pt, shift={(10,0)}] (4,-3.5)--(5,-3.5);
\node[right] at (15,-3.5) {$\alpha_{a+1}'$};
\foreach \numpt in {-0.5,3,4,5,6}{\fill  (14,\numpt-3) circle(4pt);}
\fill  (15,-3.5) circle(4pt);

\foreach \numpt in {0,1,2,3,-1,-2,-3,-4,-5,-6}{\fill (10,0) ++ (22.5*\numpt:3) circle(4pt);}

\node[below] at (11,2.7) {$m$};
\node[below left] at (12.2,2.2) {$1$};
\node[right] at (12.8,1.3) {$2$};
\node[right] at (13,0) {$3$};
\node[right] at (12.8,-1.3) {$4$};
\node[right] at (12,-2.5) {$5$};
\node[right] at (11,-3.2) {$6$};
\node[below] at (10,-3) {$7$};
\node[left] at (9,-3.2) {$8$};
\node[left] at (8,-2.5) {$9$};

\foreach \numpt in {-4,-5,-5.5,-6}{\fill (10,0) ++ (4,\numpt+3) circle(3pt);}

\foreach \numpt in {2,3}{\fill (10,0) ++ (22.5*\numpt:4.2) circle(4pt);}

\centrearc[red, line width=2pt, ->](10,0)(135:45:1)

\end{tikzpicture}
\caption{Construction \ref{con:noddmisgeq10} when $i=-1$ and $i=1$ respectively}\label{fig:noddmgeq10}
\end{centering}
\end{figure}

\begin{prop}\label{prop:noddmgeq10}
Let $m$ be even such that $10 \leq m < n$ where $n$ is odd and let $n$, $r$ and $t$ be defined according to Construction \ref{con:noddmisgeq10}. Then $\mathcal{M}(G:r,s)$ is a chiral map of type $\{ m,n \}$ and $G$ is an alternating group.
\end{prop}

\begin{proof}
Each generator is by construction an even permutation and $r$ has order $m$ while $t$ is an involution.
The permutation $s$ is a single cycle (of length $n$) with either three or four fixed points, and so the corresponding map $\mathcal{M}(G:r,s)$ is of type $\{ m, n \}$.
So long as we can prove $G$ is primitive then we may call on Jones' generalisation of Jordan's Theorem \ref{thm:JonesAltSym} (and so also unfortunately on the classification of finite simple groups) to confirm the claim that the group $G$ is alternating.

In both cases, the point labelled $m'$ is fixed by both $r$ and also by $tr^2tr^{(-3)}t$. Notice that the latter permutation sends $\beta_j$ to $\alpha_j$, and $\alpha_1$ to $1$ and, when $i=1$ it sends $\alpha'_{a+1}$ to $\alpha_{a+1}$. Combined with the definition of $r$ one can see that $Stab_{G}(m')$ is transitive on all other points, and so the group $G$ is primitive. Since $G$ is generated by even permutations, we may now conclude the group is the alternating group of the corresponding degree.

The map is chiral by the absence of reflective symmetry in the diagrams of Figure \ref{fig:noddmgeq10}. In particular when $i=-1$ there is a unique point ($2$) fixed by $r^7(tr^{-1})^3t$, and when $i=1$ the point labelled $2$ is the unique point fixed by $r^5(tr^{-1})^2t$. In each case this determines how an inverting automorphism $\phi$ (mapping $r$ to its inverse and fixing $t$) must behave, if indeed it exists, which we now assume. Bearing in mind that the group $G = \langle r,t \rangle$ is alternating, there would be an associated relabelling $\pi_\phi$ of the point labelled $2$: respectively $2 \pi_\phi=9$ or $2 \pi_\phi=7$. Note that the point $2r^{-1}$ is not fixed by $t$ in either case. However, $2 \pi_\phi r$ is fixed by $t$ in all cases except when $m=10$ and $i=-1$.
In this special case the point $2r^{-1}t = 1'$ is not fixed by $r$ while $9rt=10'$ is fixed by $r$. In every case there is a contradiction: such a relabelling $\pi_\phi$ cannot correspond to an automorphism of the permutation group $G$ which fixes $t$ and inverts $r$. We conclude that the maps resulting from Construction \ref{con:noddmisgeq10} are chiral.
\end{proof}

\subsection{Even $m > n$ odd}

We apply a different approach to this situation, remembering that if $s$ is the product of any number of distinct $n$-cycles then it will be an even permutation. Then we must ensure that we can build in an even involution $t$ such that their product $st$ has order $m$ and the permutation diagram is connected. By ensuring the generators are even permutations, the cycle structure for the product $st$ in each of our constructions turns out to be of the form $2.m$ or $2^3.m$. The latter occurs when $s$ is defined as having more than one $n$-cycle, and the constructions are then such that $s^2t$ is a single cycle fixing three points, allowing us to apply Theorem \ref{thm:JonesAltSym}.

The situation for small values of $n$ are thus treated differently. Remember that Bujalance, Conder and Costa's work, Theorem \ref{thm:BCC}, means that the case when $n=3$ need not be addressed. The next smallest odd value is when $n=5$, and the following construction gives an example of an alternating chiral map for all but finitely many types $\{ m, 5 \}$.

\begin{constr}\label{con:n=5}For $n=5$ and even $m$ such that $m +6 = 5 \nu + a $ where $\nu \geq 4$ and $0 \leq a \leq 4$.
When $1 \leq a \leq 4$ define $t_a = \prod_{j=1}^{a}(5 \nu +1 - j,5 \nu +1 -j')$ and when $a=0$ define $t_a$ to be the identity.
Meanwhile, let $t_\nu := \prod_{i=1}^{\nu - 1}(5i,5i+1)$ and define $t := (2,4)(7,9)(12,14)t_a t_\nu$ and

\noindent $s := \prod_{i=0}^{\nu -1}(5i+1,5i+2,5i+3,5i+4,5i+5)$.
\end{constr}

\begin{prop}\label{prop:mgeqnequals5}
Let $n=5$ and $m, r, t$ be defined by Construction \ref{con:n=5}. Then $G = \langle r, t \rangle $ is alternating and the map $\mathcal{M}(G ; r,s)$ is chiral of type $\{m,5 \}$.
\end{prop}

\begin{proof}
Certainly $s$ has order five, $t$ is an involution, and the permutation

\noindent $st = (1,4,6,9,11,14, \dots 5i+1, 5i+2, 5i+3, 5i+4. 5i+6 \dots 5\nu -4, \Pi, 5\nu -5, \dots, 5i, \dots 15, 10, 5)(2,3)(7,8)(12,13) $,
where $\Pi$ is the natural subsequence of $4+a$ points from the ordered list $5\nu -3',5\nu -3, 5\nu-2', 5\nu -2, 5\nu-1', 5\nu -1, 5\nu', 5\nu$.
This has (even) order $5 \nu +a -6$,
so $r=ts^{-1}$ also has order $m$,
and thus $\mathcal{M}(G:r,s)$ is a map of type $\{ m,5 \}$.

It can be checked that the permutation $s^2t$ has three fixed points, namely $2$, $7$, and $12$, and a single long cycle, as does its conjugate $(s^2t)^{((ts)^4)}$. With only one fixed point in common, we have that the stabiliser of the point labelled $2$ is transitive and so the group $G$ is primitive. We then may use the same element $s^2t$, which has a single cycle and three fixed points, to apply Jones' Theorem \ref{thm:JonesAltSym} and conclude the group is alternating. Hence $G = A_k$ where $k = 5\nu+a$ is the degree of the defined permutation group.

Chirality of the map $\mathcal{M}(G:r,s)$ can be confirmed by considering the point labelled $1$ which is unique in being fixed by both $t$ and also by $sts^2$. In contrast, there is no point which is fixed by $t$ and also by $s^{-1}ts^{-2}$ and, since $Aut(\mathcal{M}) = G = A_k$ and therefore $Aut(G) = S_k$, there can be no group automorphism fixing $t$ and inverting $s$.
\end{proof}

When $m = n+a$ with $n \geq 7$ and $1 \leq a \leq n-6$ there is an easy construction which covers this infinitude of cases and whose diagram is as shown in Figure \ref{fig:ngeq7}: $s$ is a single $n$-cycle, while $t$ consists of $a+2$ pendant edges from points on the cycle, and one edge transposing two points on the cycle. To formalise, we define the following construction.

\begin{constr}\label{con:ngeq7} For even $m=n+a$ when odd $n \geq 7$ and $1 \leq a \leq n-6$.

Define $s := (1,2,3, \dots , n-1,n)$ and $t := (1,3)\prod_{j=4}^{5+a}(j,j')$.
\end{constr}

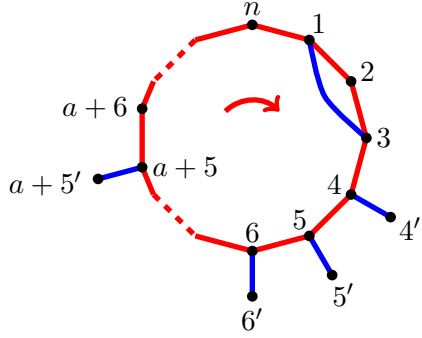
\begin{figure}

\begin{centering}
\begin{tikzpicture}[scale=0.5]

\foreach \numpt in {0,1,2,3,-1,-2,-3,-4}{\draw[red, line width=2pt] (30*\numpt:3)--(30*\numpt+30:3);}

\foreach \numpt in {4,-5}{\draw[red, dashed, line width=2pt] (30*\numpt:3)--(30*\numpt+30:3);}

\draw[blue, line width=2pt] plot [smooth] coordinates {(60:3) (30:2.2) (0:3)};

\draw[blue, line width=2pt] (-30:3)--(-30:4.2);
\draw[blue, line width=2pt] (-60:3)--(-60:4.2);
\draw[blue, line width=2pt] (-90:3)--(-90:4.2);

\draw[red, line width=2pt] (165:3)--(150:3);
\draw[red, line width=2pt] (195:3)--(-150:3);

\draw[red, line width=2pt] (165:3)--(195:3);
\draw[blue, line width=2pt] (195:3)--(195:4.2);
\fill (165:3) circle(4pt);
\fill (195:3) circle(4pt);
\fill (195:4.2) circle(4pt);

\foreach \numpt in {0,1,2,3,-1,-2,-3}{\fill (30*\numpt:3) circle(4pt);}

\foreach \numpt in {-1,-2,-3}{\fill (30*\numpt:4.2) circle(4pt);}

\node[above] at (0,3) {$n$};
\node[] at (-60:4.8) {$5'$};
\node[] at (60:3.5) {$1$};
\node[] at (30:3.5) {$2$};
\node[right] at (3,0) {$3$};
\node[] at (-30:4.8) {$4 '$};

\node[] at (-30:2.5) {$4$};
\node[] at (-60:2.5) {$5$};
\node[] at (-90:2.5) {$6$};
\node[] at (-90:4.8) {$6'$};

\node[left] at (-4.2,-1.2) {$a+5'$};
\node[right] at (-2.9,-0.8) {$a+5$};
\node[left] at (-3,0.8) {$a+6$};

\centrearc[red, line width=2pt, ->](0,0)(135:45:1)

\end{tikzpicture}
\caption{Construction \ref{con:ngeq7} when odd $n \geq 7$}\label{fig:ngeq7}
\end{centering}
\end{figure}

\begin{prop}\label{prop:mgeqngeq7}
Let $n \geq 7$ with $n$, $m$, $r$ and $t$ being defined according to Construction \ref{con:ngeq7}. Then $\mathcal{M}(G:r,s)$ is a chiral map of type $\{ m,n \}$ and $G$ is an alternating group.
\end{prop}

\begin{proof}
The order of $r^{-1}=st$ is $m = n+a$, which is even and as expected. The group generated is transitive and also primitive since $s^2t$ is a cycle with a single fixed point. Both $s$ and $t$ are even permutations, and $s$ is a single cycle with enough fixed points ($a+2 \geq 3$) to allow us to apply Jones' version of Jordan's theorem. Thus Theorem \ref{thm:JonesAltSym} proves the group is alternating. Note that in any case $t$ fixes the point labelled $n$ and this ensures the diagram in Figure \ref{fig:ngeq7} has no reflective symmetry. The resulting map is chiral by Lemma \ref{lem:chiraldiag} with $b=2$ and $c=-1$.
\end{proof}

There is a further construction which covers many more cases as follows:

\begin{constr}\label{con:ngeq11} For even $m +6 = \nu n + a $ with odd $n \geq 11$, $\nu \geq 2$ and $ 0 \leq a \leq n-1$. When $1 \leq a \leq n-1$ define $t_a = \prod_{j=1}^{a}(j,j')$ and when $a=0$ define $t_a$ to be the identity.
Meanwhile, let $t_\nu := \prod_{i=1}^{\nu - 1}(in,in+1)$ and define

\noindent $s := (1,2, \dots , n-1,n)(n+1,n+2, \dots, 2n) \dots ((\nu -1)n+1, \dots , \nu n-1, \nu n)$ and 
$t := (n+2,n+4)(n+5,n+7)(n+8,n+10)t_a t_\nu$.
\end{constr}

A permutation diagram showing an example of Construction \ref{con:ngeq11} for the case when $a=5$ and $\nu = 3$ is shown in Figure \ref{fig:ngeq11}.

\begin{figure}

\begin{centering}
\begin{tikzpicture}[scale=0.5]

\foreach \numpt in {0,1,2,3,4,-1,-2,-3,-4,-5,-6,-7}{\draw[red, line width=2pt] (22.5*\numpt:3)--(22.5*\numpt+22.5:3);}

\centrearc[red, line width=2pt, dashed](0,0)(112.5:202.5:3)

\draw[blue, line width=2pt] (67.5:3)--(67.5:4.2);
\draw[blue, line width=2pt] (45:3)--(45:4.2);
\draw[blue, line width=2pt] (22.5:3)--(22.5:4.2);
\draw[blue, line width=2pt] (0:3)--(0:4.2);
\draw[blue, line width=2pt] (-22.5:3)--(-22.5:4.2);

\foreach \numpt in {0,1,2,3,4,-1,-2,-3,-4,-5,-6}{\fill (22.5*\numpt:3) circle(4pt);}

\foreach \numpt in {-1,0,1,2,3}{\fill (22.5*\numpt:4.2) circle(4pt);}

\centrearc[red, line width=2pt, ->](0,0)(135:45:1)

\foreach \numpt in {0,1,2,3,4,-1,-2,-3,-4,-5,-6,-7}{\draw[red, line width=2pt, shift={(-8,0)}] (22.5*\numpt:3)--(22.5*\numpt+22.5:3);}

\centrearc[red, line width=2pt, dashed, shift={(-8,0)}](0,0)(112.5:202.5:3)

\centrearc[red, line width=2pt, dashed, shift={(-16,0)}](0,0)(112.5:202.5:3)

\draw[blue, line width=2pt,] plot [smooth] coordinates {(0,3) (-3.5,4) (-7,2.7)};

\draw[blue, line width=2pt, shift={(-8,0)}] plot [smooth] coordinates {(0:3) (22.5:2.2) (45:3)};

\draw[blue, line width=2pt, shift={(-8,0)}] plot [smooth] coordinates {(-22.5:3) (-45:2.2) (-67.5:3)};

\draw[blue, line width=2pt, shift={(-8,0)}] plot [smooth] coordinates {(-90:3) (-112.5:2.2) (-135:3)};

\foreach \numpt in {0,1,2,3,4,-1,-2,-3,-4,-5,-6}{\fill[shift={(-8,0)}] (22.5*\numpt:3) circle(4pt);}

\centrearc[red, line width=2pt, ->](-8,0)(135:45:1)

\draw[blue, line width=2pt,] plot [smooth] coordinates {(-8,3) (-11.5,4) (-15,2.7)};

\foreach \numpt in {0,1,2,3,4,-1,-2,-3,-4,-5,-6,-7}{\draw[red, line width=2pt, shift={(-16,0)}] (22.5*\numpt:3)--(22.5*\numpt+22.5:3);}

\foreach \numpt in {0,1,2,3,4,-1,-2,-3,-4,-5,-6}{\fill[shift={(-16,0)}] (22.5*\numpt:3) circle(4pt);}

\centrearc[red, line width=2pt, ->](-16,0)(135:45:1)

\node[] at (67.5:2.5) {$1$};
\node[] at (45:2.5) {$2$};
\node[] at (22.5:2.5) {$3$};
\node[] at (-22.5:2.5) {$5$};
\node[] at (-45:2.5) {$6$};
\node[] at (-67.5:2.5) {$7$};
\node[] at (-90:2.5) {$8$};
\node[] at (-112.5:2.5) {$9$};
\node[] at (-135:2.5) {$10$};

\node[] at (67.5:4.8) {$1'$};
\node[] at (45:4.8) {$2'$};
\node[] at (22.5:4.8) {$3'$};
\node[] at (-22.5:4.8) {$5'$};

\node[below] at (0,3) {$n$};
\node[below] at (-8,3) {$2n$};
\node[below] at (-16,3) {$3n$};

\node[left] at (3,0) {$4$};
\node[right] at (4.2,0) {$4 '$};

\node[] at (-4.7,2.2) {$n+2$};
\node[] at (-4,1.1) {$n+3$};
\node[] at (-4.7,-2.2) {$n+6$};

\node[] at (-12.6,2.2) {$2n+2$};
\node[] at (-12,1.1) {$2n+3$};
\node[] at (-12.6,-2.2) {$2n+6$};

\end{tikzpicture}
\caption{Example of Construction \ref{con:ngeq11} when $n \geq 11$, $a=5$, $\nu=3$}\label{fig:ngeq11}
\end{centering}
\end{figure}
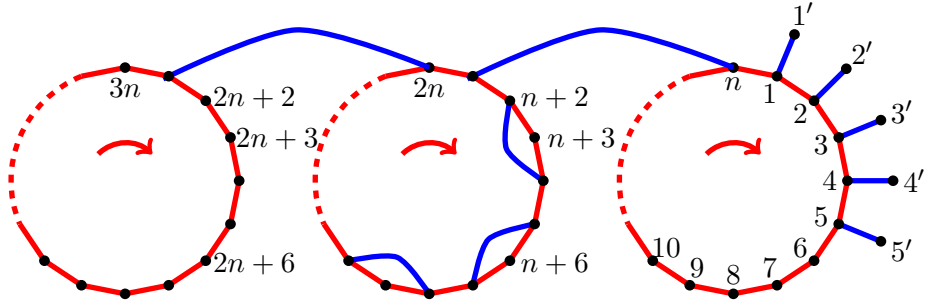

\begin{prop}\label{prop:mgeqngeq11}
Let $n \geq 11$ be odd, with even $m$ and the permutations $s$ and $t$ being defined by Construction \ref{con:ngeq11}. Then $\mathcal{M}(G:r,s)$ is a chiral map of type $\{ m,n \}$ and $G$ is an alternating group.
\end{prop}

\begin{proof}
When $\nu$ is even $a$ must be even so $t_a$ is even, and $t_\nu$ is odd. Meanwhile, if $\nu$ is odd then $t_a$ is odd, and $t_\nu$ is even. In any case the element $t$ is an even permutation, as is $s$. The order of $r^{-1}=st$ is $\nu n+a-6=m$ and so the map has type $\{ m,n \}$.

The group is clearly transitive by connectivity of the diagram, and primitivity is addressed by similar arguments to previously: $s^2t$ has three fixed points and a long cycle, as does its conjugate $(s^2t)^{s^6}$. With only one fixed point in common, and since $\langle s^2t, (s^2t)^{s^6} \rangle$ acts transitively on all other points, we have that the stabiliser of the point labelled $n+2$ is transitive and so the group is primitive. We use the same element $s^2t$ to apply Jones' Theorem \ref{thm:JonesAltSym} and conclude the group $G$ is alternating.

We may then apply the usual reasoning with reference to the permutation diagram to prove chirality. Under the action of a hypothetical inverting automorphism the unique points $n+2$ and $n+10$ would have to be interchanged. The only way in which this could be consistent with an orientation-reversing relabelling of the diagram would be if $n=11$ and $a=0$ and $\nu=3$ but this would contradict our ongoing assumption that $m$ is even.
\end{proof}

The above two constructions cover all types where even $m > n \geq 11$. For smaller values of $n$ just a single $n$-cycle is insufficient for the structures upon which we rely, namely three fixed points of $s^2t$. The following construction applies to the two remaining values of $n$, namely $n=7$ and $n=9$, for which we require at least two $n$-cycles in order to accommodate the three fixed points for $s^2t$, and a further $n$-cycle to allow for the congruence class of $m+6 \equiv a$ modulo $n$.

\begin{constr}\label{con:n=7or9} For $n \in \{ 7, 9\}$.
Let even $m + 6 = \nu n + a$ with $\nu \geq 3$ and $0 \leq a \leq n-1$.
Let $t_\nu =\prod_{i=1}^{\nu-1}(n_i, n_i +1)$. When $a \geq 1$ let $t_a =\prod_{j=1}^{a}(n\nu +1 - j, n\nu +1 -j ')$, and trivial otherwise.
Define 

$s := \prod_{i=0}^{\nu-1}(n i +1, n i +2, \dots , ni+n)$ and $t : = (13)(46)(n+2,n+4)t_\nu t_a$.

\end{constr}

\begin{prop}\label{prop:mgeqnequals7or9}
Let fixed $n \in \{ 7, 9\}$ with $m$, $s$ and $t$ being defined according to Construction \ref{con:n=7or9}. Then $\mathcal{M}(G:r,s)$ is a chiral map of type $\{ m,n \}$ and $G$ is an alternating group.
\end{prop}

\begin{proof}
Omitted - possible using similar arguments to previous cases. 
\end{proof}
\bigskip

\section{When both $m$ and $n$ are even}\label{sec:neven}

When both $m$ and $n$ are even we may rely on duality and assume $m \leq n$.
The constructions presented in the previous section can then be modified for our purposes. An advantage of having both parameters even is that we may, without an unwanted consequential change in the order, allow transpositions to occur in the products $rt$ and/or $st$. We must still be careful to ensure that the orders are as we require, and that the generators are even permutations.

\begin{constr}\label{con:nevenmis4} For $m = 4$. See Figure \ref{fig:nevenmis4}.

Let $n = 4a+7+i$ where $i \in \{ -1, 1 \}$ and $a \geq 1$. Define $s := s_i$ where

\noindent  $s_{-1} := (1,2,\dots ,2a-1,2a,\alpha,\alpha',\beta,2a',2a-1', \dots ,2',1',\gamma, \delta, \gamma')(\delta',\epsilon)$ and

\noindent $s_{1} := (1,2, \dots ,2a-1,2a,\alpha,\alpha',\beta,\beta',\zeta,2a', 2a-1', \dots ,2',1',\gamma, \delta, \gamma')(\delta',\epsilon)$.

\noindent Define $t : = (\alpha,\alpha ')(\beta,\beta')(\gamma,\gamma')(\delta,\delta')\prod_{j=1}^{2a}(j,j')$.
\end{constr} 

\begin{figure}

\begin{centering}
\begin{tikzpicture}[scale=0.4]

\foreach \numpt in {0,1,2,3}{\draw [red, line width=2pt]  (90+45*\numpt:3)--(90+45+45*\numpt:3);}

\draw [red, line width=2pt]  (0,3)--(1.5,3);
\draw [red, line width=2pt]  (0,-3)--(1.5,-3);
\draw [red, line width=2pt]  (2.5,3)--(4,3);
\draw [red, line width=2pt]  (2.5,-3)--(4,-3);

\draw[blue, line width=2pt] plot [smooth] coordinates {(135:3)  (180:1.8)  (-135:3)};

\draw [blue, line width=2pt]  (180:3)--(180:4.2);
\draw [red, line width=2pt]  (180:5.4)--(180:4.2);
\fill  (180:4.2) circle(4pt);
\fill  (180:5.4) circle(4pt);

\draw [blue, line width=2pt]  (0,3)--(0,-3);
\draw [blue, line width=2pt]  (1,3)--(1,-3);

\foreach \numpt in {0,1,2,3,4}{\fill  (90+45*\numpt:3) circle(4pt);}

\draw [blue, line width=2pt]  (4,3)--(4,-3);
\draw [blue, line width=2pt]  (3,3)--(3,-3);

\draw [blue, line width=2pt, shift={(4,0)}]  (-45:3)--(-45:4.2);

\foreach \numpt in {-2,-1,0,1}{\draw [red, line width=2pt, shift={(4,0)}]  (45*\numpt:3)--(45+45*\numpt:3);}

\draw[blue, line width=2pt, shift={(4,0)}] plot [smooth] coordinates {(45:3)  (22.5:1.8)  (0:3)};

\centrearc[red, line width=2pt, ->](4,0)(45:-45:1)

\foreach \numpt in {-2,-1,0,1,2}{\fill[shift={(4,0)}]  (45*\numpt:3) circle(4pt);}
\fill[shift={(4,0)}]  (-45:4.2) circle(4pt);

\node[above] at (4,3) {$2a$};
\node[below] at (4,-3) {$2a'$};

\node[above right] at (6,2) {$\alpha$};
\node[right] at (7,0) {$\alpha '$};
\node[above left] at (6.2,-2.4) {$\beta $};
\node[below] at (7,-3) {$\beta '$};

\node[above] at (0,3) {$1$};
\node[below] at (0,-3) {$1'$};

\node[] at (135:3.4) {$\gamma '$};
\node[] at (-135:3.6) {$\gamma$};

\node[right] at (180:3) {$\delta$};
\node[above] at (180:4.1) {$\delta '$};
\node[above] at (180:5.4) {$\epsilon$};

\draw [red, dotted, line width=2pt]  (1.5,3)--(2.5,3);
\draw [red, dotted, line width=2pt]  (1.5,-3)--(2.5,-3);
\end{tikzpicture}
\hfill{}
\begin{tikzpicture}[scale=0.4]

\draw [red, dotted, line width=2pt]  (1.5,3)--(2.5,3);
\draw [red, dotted, line width=2pt]  (1.5,-3)--(2.5,-3);

\foreach \numpt in {0,1,2,3}{\draw [red, line width=2pt]  (90+45*\numpt:3)--(90+45+45*\numpt:3);}

\draw [red, line width=2pt]  (0,3)--(1.5,3);
\draw [red, line width=2pt]  (0,-3)--(1.5,-3);
\draw [red, line width=2pt]  (2.5,3)--(4,3);
\draw [red, line width=2pt]  (2.5,-3)--(4,-3);

\draw[blue, line width=2pt] plot [smooth] coordinates {(135:3)  (180:1.8)  (-135:3)};

\draw [blue, line width=2pt]  (180:3)--(180:4.2);
\draw [red, line width=2pt]  (180:5.4)--(180:4.2);
\fill  (180:4.2) circle(4pt);
\fill  (180:5.4) circle(4pt);

\draw [blue, line width=2pt]  (0,3)--(0,-3);
\draw [blue, line width=2pt]  (1,3)--(1,-3);

\foreach \numpt in {0,1,2,3,4}{\fill  (90+45*\numpt:3) circle(4pt);}

\draw [blue, line width=2pt]  (4,3)--(4,-3);
\draw [blue, line width=2pt]  (3,3)--(3,-3);

\foreach \numpt in {-3,-2,-1,0,1,2}{\draw [red, line width=2pt, shift={(4,0)}]  (30*\numpt:3)--(30+30*\numpt:3);}

\draw[blue, line width=2pt, shift={(4,0)}] plot [smooth] coordinates {(60:3)  (45:2)  (30:3)};
\draw[blue, line width=2pt, shift={(4,0)}] plot [smooth] coordinates {(0:3)  (-15:2)  (-30:3)};

\centrearc[red, line width=2pt, ->](4,0)(45:-45:1)

\foreach \numpt in {-3,-2,-1,0,1,2,3}{\fill[shift={(4,0)}]  (30*\numpt:3) circle(4pt);}

\node[above] at (4,3) {$2a$};
\node[below] at (4,-3) {$2a'$};

\node[above right] at (5.5,2.4) {$\alpha $};
\node[right] at (6.6,2) {$\alpha '$};
\node[right] at (7,0) {$\beta$};

\node[right] at (6.6,-2) {$\beta '$};
\node[below, right] at (5.5,-3) {$\zeta $};

\node[above] at (0,3) {$1$};
\node[below] at (0,-3) {$1'$};

\node[] at (135:3.4) {$\gamma '$};
\node[] at (-135:3.6) {$\gamma$};

\node[right] at (180:3) {$\delta$};
\node[above] at (180:4.1) {$\delta '$};
\node[above] at (180:5.4) {$\epsilon$};

\end{tikzpicture}
\caption{Construction \ref{con:nevenmis4} when $i=-1$ and $i=1$ respectively}\label{fig:nevenmis4}
\end{centering}
\end{figure}
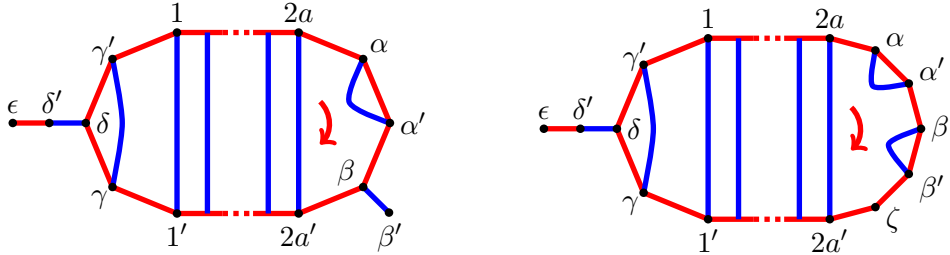

\begin{prop}\label{prop:nevenmis4}
Let $n$, $s$ and $t$ be defined according to Construction \ref{con:nevenmis4}. Then $\mathcal{M}(G:r,s)$ is a chiral map of type $\{ 4,n \}$ and $G = A_{k}$.
\end{prop}

\begin{proof}
The type of each map is easy to check, and is $\{ 4, n \}$ as expected.
There is precisely one point (labelled $\epsilon$) which is fixed by $s^2$ and $t$. These two permutations act transitively on all other points ($t$ links the three other orbits of $s^2$) and so the group $G$ is primitive.
When $i=-1$, $(s^4t)^6$ is a $7$-cycle, and when $i=1$ the permutation $(s^3t)^{12}$ is a $5$-cycle,
so by Jordan's Theorem \ref{thm:JordanAltSym} in each case the group $G = \langle s,t \rangle$ is alternating.
The maps are chiral by inspection of the diagrams in Figure \ref{fig:nevenmis4} using respectively Lemma \ref{lem:chiraldiagsfixedpoint} with $b=1$ and $c=2$ and Lemma \ref{lem:chiraldiag} with $b=4$ and $c=-2a$.
\end{proof}

Notice that this construction may be extended to the cases when $a=0$, that is when $n \in \{ 6, 8 \}$, by omitting all the numbered points and their primes. This gives $t=(\alpha,\alpha ')(\beta,\beta')(\gamma, \gamma')(\delta, \delta')$ while $s_{-1}=(\alpha, \alpha', \beta, \gamma, \delta, \gamma')(\delta',\epsilon)$ and $s_1 = (\alpha, \alpha', \beta,\beta', \zeta, \gamma, \delta, \gamma')(\delta',\epsilon)$. Respectively these yield a map of type $\{ 4, 6 \}$ with simple underlying group $G = \langle s,t \rangle = A_9$, and a map of type $\{ 4, 8 \}$ with $G = A_{10}$. Chirality in each case may be proved in the same way as in the proof of Proposition \ref{prop:nevenmis4}.

Remembering that the type $\{4,4 \}$ is toroidal, the work so far proves there is a simple (alternating) group supporting a chiral map for any given hyperbolic type $\{ 4,n \}$ and, so also by duality for any hyperbolic type $\{n,4 \}$.

A minor modification of the diagram for $m=4$ yields the next construction, valid for when $m=6$.

\begin{constr}\label{con:nevenmis6} For $m = 6$. See Figure \ref{fig:nevenmis6}.

Let $n = 4a+7+i$ where $i \in \{ -1, 1 \}$ and $a \geq 1$. Define $s := s_i$ where

\noindent $s_{-1} = (1,2, \dots ,2a-1,2a,\alpha,\alpha',\beta,2a',2a-1', \dots ,2',1',\delta,\epsilon, \delta')(\beta', \gamma)$,

\noindent $s_{1} = (1,2, \dots , 2a-1,2a,\alpha,\alpha',\beta,\beta',\gamma,2a', 2a-1', \dots ,2',1',\delta,\epsilon, \delta')(\gamma', \zeta)$.

\noindent Define $t : = (\alpha,\alpha ')(\beta,\beta')(\gamma,\gamma')(\delta,\delta')\prod_{j=1}^{2a}(j,j')$.
\end{constr}

\begin{figure}

\begin{centering}
\begin{tikzpicture}[scale=0.4]

\foreach \numpt in {0,1,2,3}{\draw [red, line width=2pt]  (90+45*\numpt:3)--(90+45+45*\numpt:3);}

\draw [red, line width=2pt]  (0,3)--(1.5,3);
\draw [red, line width=2pt]  (0,-3)--(1.5,-3);
\draw [red, line width=2pt]  (2.5,3)--(4,3);
\draw [red, line width=2pt]  (2.5,-3)--(4,-3);

\draw[blue, line width=2pt] plot [smooth] coordinates {(135:3)  (180:1.8)  (-135:3)};

\draw [blue, line width=2pt]  (0,3)--(0,-3);
\draw [blue, line width=2pt]  (1,3)--(1,-3);

\foreach \numpt in {0,1,2,3,4}{\fill  (90+45*\numpt:3) circle(4pt);}

\draw [blue, line width=2pt]  (4,3)--(4,-3);
\draw [blue, line width=2pt]  (3,3)--(3,-3);

\foreach \numpt in {-2,-1,0,1}{\draw [red, line width=2pt, shift={(4,0)}]  (45*\numpt:3)--(45+45*\numpt:3);}

\draw[blue, line width=2pt, shift={(4,0)}] plot [smooth] coordinates {(45:3)  (22.5:1.8)  (0:3)};

\draw [blue, line width=2pt, shift={(4,0)}]  (-45:3)--(-45:4.2);
\draw [red, line width=2pt, shift={(4,0)}]  (-45:5.4)--(-45:4.2);
\draw [blue, line width=2pt, shift={(4,0)}]  (-45:5.4)--(-45:6.6);

\centrearc[red, line width=2pt, ->](4,0)(45:-45:1)

\foreach \numpt in {-2,-1,0,1,2}{\fill[shift={(4,0)}]  (45*\numpt:3) circle(4pt);}
\fill[shift={(4,0)}]  (-45:4.2) circle(4pt);
\fill[shift={(4,0)}]  (-45:5.4) circle(4pt);
\fill[shift={(4,0)}]  (-45:6.6) circle(4pt);

\node[above] at (4,3) {$2a$};
\node[below] at (4,-3) {$2a '$};

\node[above right] at (6,2) {$\alpha$};
\node[right] at (7,0) {$\alpha '$};
\node[above left] at (6.2,-2.4) {$\beta $};

\node[right] at (7,-2.7) {$\beta '$};
\node[right] at (7.75,-3.6) {$\gamma$};
\node[right] at (8.5,-4.5) {$\gamma '$};

\node[above] at (0,3) {$1$};
\node[below] at (0,-3) {$1'$};

\node[] at (135:3.4) {$\delta '$};
\node[] at (180:3.5) {$\epsilon$};
\node[] at (-135:3.5) {$\delta$};

\draw [red, dotted, line width=2pt]  (1.5,3)--(2.5,3);
\draw [red, dotted, line width=2pt]  (1.5,-3)--(2.5,-3);
\end{tikzpicture}
\hfill{}
\begin{tikzpicture}[scale=0.4]

\draw [red, dotted, line width=2pt]  (1.5,3)--(2.5,3);
\draw [red, dotted, line width=2pt]  (1.5,-3)--(2.5,-3);

\foreach \numpt in {0,1,2,3}{\draw [red, line width=2pt]  (90+45*\numpt:3)--(90+45+45*\numpt:3);}

\draw [red, line width=2pt]  (0,3)--(1.5,3);
\draw [red, line width=2pt]  (0,-3)--(1.5,-3);
\draw [red, line width=2pt]  (2.5,3)--(4,3);
\draw [red, line width=2pt]  (2.5,-3)--(4,-3);

\draw[blue, line width=2pt] plot [smooth] coordinates {(135:3)  (180:1.8)  (-135:3)};

\draw [blue, line width=2pt]  (0,3)--(0,-3);
\draw [blue, line width=2pt]  (1,3)--(1,-3);

\foreach \numpt in {0,1,2,3,4}{\fill  (90+45*\numpt:3) circle(4pt);}

\draw [blue, line width=2pt]  (4,3)--(4,-3);
\draw [blue, line width=2pt]  (3,3)--(3,-3);

\foreach \numpt in {-3,-2,-1,0,1,2}{\draw [red, line width=2pt, shift={(4,0)}]  (30*\numpt:3)--(30+30*\numpt:3);}

\draw [blue, line width=2pt, shift={(4,0)}]  (-60:3)--(-60:4.2);
\draw [red, line width=2pt, shift={(4,0)}]  (-60:5.4)--(-60:4.2);
\fill[shift={(4,0)}] (-60:4.2) circle(4pt);
\fill[shift={(4,0)}] (-60:5.4) circle(4pt);

\draw[blue, line width=2pt, shift={(4,0)}] plot [smooth] coordinates {(60:3)  (45:2)  (30:3)};
\draw[blue, line width=2pt, shift={(4,0)}] plot [smooth] coordinates {(0:3)  (-15:2)  (-30:3)};

\centrearc[red, line width=2pt, ->](4,0)(45:-45:1)

\foreach \numpt in {-3,-2,-1,0,1,2,3}{\fill[shift={(4,0)}]  (30*\numpt:3) circle(4pt);}

\node[above] at (0,3) {$1$};
\node[below] at (0,-3) {$1'$};

\node[] at (135:3.4) {$\delta '$};
\node[] at (180:3.5) {$\epsilon$};
\node[] at (-135:3.5) {$\delta$};

\node[above] at (4,3) {$2a$};
\node[below] at (4,-3) {$2a'$};

\node[above right] at (5.5,2.4) {$\alpha $};
\node[right] at (6.6,2) {$\alpha '$};
\node[right] at (7,0) {$\beta$};

\node[right] at (6.6,-2) {$\beta '$};
\node[above] at (5.3,-2.8) {$\gamma $};

\node[right] at (6,-3.5) {$\gamma '$};
\node[right] at (6.6,-4.7) {$\zeta $};

\end{tikzpicture}
\caption{Construction \ref{con:nevenmis6} when $i=-1$ and $i=1$ respectively}\label{fig:nevenmis6}
\end{centering}
\end{figure}
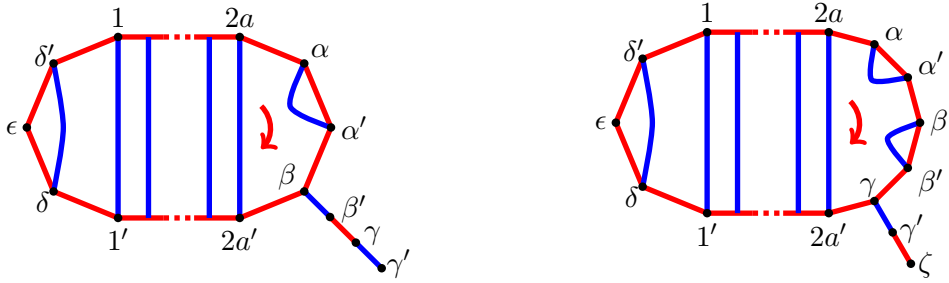

\begin{prop}\label{prop:nevenmis6}
Let $n$, $s$ and $t$ be defined as in Construction \ref{con:nevenmis6}. Then $\mathcal{M}(G:r,s)$ is a chiral map of type $\{ 6,n \}$ and $G = A_k$.
\end{prop}

\begin{proof}
The map has the correct type $\{ 6, n \}$ since $s$ is an $n$-cycle and $r = ts^{-1}$ is the product of transpositions with a $6$-cycle.
The group $G = \langle s, t \rangle$ is primitive: when $i=-1$ the point $\gamma '$ is fixed by $r$ and $tr^2t$ which together act transitively on all the other points; when $i=1$ the elements $r^2$ and $t$ fix the point $\zeta$ and they act transitively on the other points.
When $i=-1$, $(s^4t)^2$ is a $7$-cycle,
when $i=1$ the permutation $(s^2t)^6$ is a $5$-cycle and the generators are even so, by Jordan's Theorem \ref{thm:JordanAltSym}, the group $G$ is alternating of the same degree as the diagram.
The map is chiral by Lemma \ref{lem:chiraldiagsfixedpoint} using $b=2$ and $c=-2a-1$, see Figure \ref{fig:nevenmis6}.
\end{proof}

Notice that this construction can also be extended to the case when $a=0$ and $i=1$ to yield a map of type $\{ 6,8 \}$ with generators $s= (\alpha,\alpha',\beta, \beta', \gamma, \delta,\epsilon,\delta')(\gamma', \zeta)$ and $t=(\alpha,\alpha ')(\beta,\beta')(\gamma,\gamma')(\delta,\delta')$ and alternating automorphism group $G = \langle s,t \rangle = A_{10}$. Chirality may then be proved by Lemma \ref{lem:chiraldiagsfixedpoint} using $b=2$ and $c=-1$.

The remaining cases are covered by the following two constructions.

\begin{constr}\label{con:nevenmis8} For $m = 8$. See Figure \ref{fig:nevenmis8}.

Let $n = 4a+7+i \geq 8$ where $i \in \{ -1 , 1 \}$ and $a \geq 0$.
Define $s := s_i$ where
 $s_{-1} := (1,2, \dots ,2a-1,2a,\alpha,\gamma,\epsilon,2a',2a-1', \dots ,2',1',\delta,\zeta, \delta')(\alpha', \beta)$ while 
 
 \noindent $s_{1} = (\alpha,\beta,\gamma,\gamma',\epsilon,\delta,\zeta, \delta')(\alpha',\eta )$ when $a=0$, and when $a \geq 1$

\noindent $s_{1} = (1,2, \dots 2a-1,2a,\alpha,\beta,\gamma,\gamma',\epsilon,2a',2a-1', \dots ,2',1',\delta,\zeta, \delta')(\alpha',\eta )$.
 
 Define $t := t_a(\alpha,\alpha ')(\beta,\beta')(\gamma,\gamma')(\delta,\delta')$ 
where, $t_a$ is trivial when $a=0$ and otherwise $t_a := \prod_{j=1}^{2a}(j,j')$.
\end{constr}

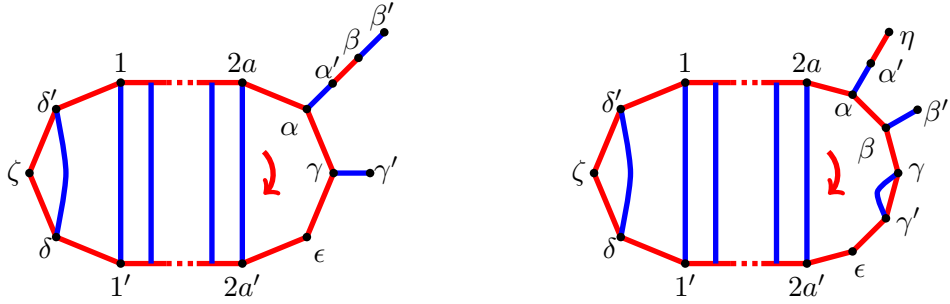
\begin{figure}

\begin{centering}
\begin{tikzpicture}[scale=0.4]

\foreach \numpt in {0,1,2,3}{\draw [red, line width=2pt]  (90+45*\numpt:3)--(90+45+45*\numpt:3);}

\draw [red, line width=2pt]  (0,3)--(1.5,3);
\draw [red, line width=2pt]  (0,-3)--(1.5,-3);
\draw [red, line width=2pt]  (2.5,3)--(4,3);
\draw [red, line width=2pt]  (2.5,-3)--(4,-3);

\draw[blue, line width=2pt] plot [smooth] coordinates {(135:3)  (180:1.8)  (-135:3)};

\draw [blue, line width=2pt]  (0,3)--(0,-3);
\draw [blue, line width=2pt]  (1,3)--(1,-3);

\foreach \numpt in {0,1,2,3,4}{\fill  (90+45*\numpt:3) circle(4pt);}

\draw [blue, line width=2pt]  (4,3)--(4,-3);
\draw [blue, line width=2pt]  (3,3)--(3,-3);
\draw [blue, line width=2pt]  (7,0)--(8.2,0);

\foreach \numpt in {-2,-1,0,1}{\draw [red, line width=2pt, shift={(4,0)}]  (45*\numpt:3)--(45+45*\numpt:3);}

\draw [blue, line width=2pt, shift={(4,0)}]  (45:3)--(45:4.2);
\draw [red, line width=2pt, shift={(4,0)}]  (45:5.4)--(45:4.2);
\draw [blue, line width=2pt, shift={(4,0)}]  (45:5.4)--(45:6.6);

\centrearc[red, line width=2pt, ->](4,0)(45:-45:1)

\foreach \numpt in {-2,-1,0,1,2}{\fill[shift={(4,0)}]  (45*\numpt:3) circle(4pt);}
\fill[shift={(4,0)}]  (45:4.2) circle(4pt);
\fill[shift={(4,0)}]  (45:5.4) circle(4pt);
\fill[shift={(4,0)}]  (45:6.6) circle(4pt);
\fill[shift={(4,0)}]  (4.2,0) circle(4pt);

\node[above] at (0,3) {$1$};
\node[below] at (0,-3) {$1'$};

\node[] at (135:3.4) {$\delta '$};
\node[] at (180:3.5) {$\zeta$};
\node[] at (-135:3.5) {$\delta$};

\node[above] at (4,3) {$2a$};
\node[below] at (4,-3) {$2a'$};

\node[below left] at (6.2,2) {$\alpha$};
\node[above] at (6.8,2.8) {$\alpha '$};
\node[above] at (7.6,3.6) {$\beta $};
\node[above] at (8.6,4.4) {$\beta '$};
\node[left] at (7,0) {$\gamma $};
\node[right] at (8,0.1) {$\gamma '$};
\node[below right] at (6,-2.2) {$\epsilon$};

\draw [red, dotted, line width=2pt]  (1.5,3)--(2.5,3);
\draw [red, dotted, line width=2pt]  (1.5,-3)--(2.5,-3);
\end{tikzpicture}
\hfill{}
\begin{tikzpicture}[scale=0.4]

\draw [red, dotted, line width=2pt]  (1.5,3)--(2.5,3);
\draw [red, dotted, line width=2pt]  (1.5,-3)--(2.5,-3);

\foreach \numpt in {0,1,2,3}{\draw [red, line width=2pt]  (90+45*\numpt:3)--(90+45+45*\numpt:3);}

\draw [red, line width=2pt]  (0,3)--(1.5,3);
\draw [red, line width=2pt]  (0,-3)--(1.5,-3);
\draw [red, line width=2pt]  (2.5,3)--(4,3);
\draw [red, line width=2pt]  (2.5,-3)--(4,-3);

\draw[blue, line width=2pt] plot [smooth] coordinates {(135:3)  (180:1.8)  (-135:3)};

\draw [blue, line width=2pt]  (0,3)--(0,-3);
\draw [blue, line width=2pt]  (1,3)--(1,-3);

\foreach \numpt in {0,1,2,3,4}{\fill  (90+45*\numpt:3) circle(4pt);}

\draw [blue, line width=2pt]  (4,3)--(4,-3);
\draw [blue, line width=2pt]  (3,3)--(3,-3);

\foreach \numpt in {-3,-2,-1,0,1,2}{\draw [red, line width=2pt, shift={(4,0)}]  (30*\numpt:3)--(30+30*\numpt:3);}

\draw [blue, line width=2pt, shift={(4,0)}]  (60:3)--(60:4.2);
\draw [blue, line width=2pt, shift={(4,0)}]  (30:3)--(30:4.2);
\draw [red, line width=2pt, shift={(4,0)}]  (60:5.4)--(60:4.2);
\fill[shift={(4,0)}] (60:4.2) circle(4pt);
\fill[shift={(4,0)}] (30:4.2) circle(4pt);
\fill[shift={(4,0)}] (60:5.4) circle(4pt);

\draw [blue, line width=2pt, shift={(4,0)}] plot [smooth] coordinates {(0:3)  (-15:2.4)  (-30:3)};

\centrearc[red, line width=2pt, ->](4,0)(45:-45:1)

\foreach \numpt in {-3,-2,-1,0,1,2,3}{\fill[shift={(4,0)}]  (30*\numpt:3) circle(4pt);}

\node[above] at (0,3) {$1$};
\node[below] at (0,-3) {$1'$};

\node[] at (135:3.4) {$\delta '$};
\node[] at (180:3.5) {$\zeta$};
\node[] at (-135:3.5) {$\delta$};

\node[above] at (4,3) {$2a$};
\node[below] at (4,-3) {$2a'$};

\node[] at (5.2,2.1) {$\alpha $};
\node[right] at (6,3.4) {$\alpha '$};
\node[right] at (6.7,4.4) {$\eta $};
\node[below left] at (6.6,1.5) {$\beta$};
\node[right] at (7.5,2) {$\beta '$};

\node[right] at (7,0) {$\gamma$};
\node[right] at (6.6,-1.6) {$\gamma '$};
\node[below right] at (5.2,-2.7) {$\epsilon$};
\end{tikzpicture}
\caption{Construction \ref{con:nevenmis8} when $i=-1$ and $i=1$ respectively}\label{fig:nevenmis8}
\end{centering}
\end{figure}

\begin{prop}\label{prop:nevenmis8}
Let $n$, $s$ and $t$ be defined according to Construction \ref{con:nevenmis8}. Then $\mathcal{M}(G:r,s)$ is a chiral map of type $\{ 8,n \}$ and $G = A_k$.
\end{prop}

\begin{proof}
The map has the correct type - this is easy to check.

For each case we suppose, for a contradiction, the group is imprimitive, and that $\mathcal{B}$ is a non-trivial proper block system.
In the case when $i=-1$ there are precisely two points which are fixed by $s$. Note that neither of the two points $\gamma '$ and $\beta'$ can be in the same block as any point on the long $s$-cycle, since that block would then necessarily include all the points, a contradiction. Also, neither of the points $\alpha'$ or $\beta$ can be in the same block as $\gamma '$, since then both would be, and the block containing $\beta '$ would have to contain $\alpha$ which is on the $s$-cycle and leads to the same contradiction. Therefore $\gamma '$ and $\beta'$ must be in the same block $B_1 \in \mathcal{B}$ and moreover the maximum size of a non-trivial block is two. Applying $t$ to the elements in $B_1$ implies that the points $\gamma$ and $\beta$ are in $B_2 \in \mathcal{B}$, and applying $s^2$ (which fixes $\beta$) leads us to the contradiction that also $2a' \in B_2$.
In the case where $i=1$ there is a unique point, $\beta' \in B_1$, which is fixed by $s$. The only plausible candidates for other elements in $B_1 \in \mathcal{B}$ are one of, and hence both, $\alpha '$ and $\eta$. Application of $tst$ then implies $B_1$ contains $\gamma '$, an element on the long $s$-cycle, and so $B_1$ contains all the points, a contradiction.
We conclude that the construction yields only primitive groups.

When $i=-1$, $(s^2t)^4$ is a $3$-cycle,
and when $i=1$ the permutation $(s^3t)^2$ is a $11$-cycle. So long as $n \neq 8$ we may use these permutations to apply Jordan's Theorem \ref{thm:JordanAltSym} and conclude that the group $G$ is alternating. In the case where $n=8$ then $k=11$ and it can be checked that the group generated is also alternating, $G = A_{11}$.
The map is then chiral by inspection of the diagrams in Figure \ref{fig:nevenmis8} using Lemma \ref{lem:chiraldiag} and $b=2$ and $c = -2a - 1$.
\end{proof}

\begin{constr}\label{con:nevenmgeq10} For $m \geq 10$. See Figure \ref{fig:nevenmgeq10}.

Let $n = m + 4a + i $ where $i \in \{ 0, 2 \}$ and $a \geq 0$.

When $m=n \equiv 0$ modulo $4$ let $r = (1,2, \dots , m-1, m)(1',\alpha,\beta, \gamma)$ and 
\noindent $t = (1,1')(2,3)(4,5)(6,8)$.

Otherwise let $r = (1,2, \dots , m-1, m)(1',\alpha_1)r_ar_i$ and 

\noindent $t = (1,1')(2,3)(4,5)(6,8)(m-1,m-1')(m,m')t_a$ where,
for $a \geq 1$ define $t_a := \Pi_{j=1}^{a}{(\alpha_j,\alpha_j')(\beta_j,\beta_j')}$ and $r_a := \Pi_{j=1}^{a}{(\alpha_j',\beta_j)(\beta_j',\alpha_{j+1})}$, and when $a=0$ we define $t_a$ and $r_a$ to be the identity. When $i=0$ let $r_i$ be the identity, otherwise let $r_i = (m-1',\gamma)(m',\delta)$.
\end{constr}

\begin{figure}

\begin{centering}\begin{tikzpicture}[scale=0.6]

\foreach \numpt in {0,1,2,3,4,-1,-2,-3,-4,-5,-6}{\draw[red, line width=2pt] (22.5*\numpt:3)--(22.5*\numpt+22.5:3);}

\centrearc[red, line width=2pt, dashed](0,0)(112.5:225:3)

\draw[blue, line width=2pt] plot [smooth] coordinates {(0:3) (11.25:2.2) (22.5:3)};

\draw[blue, line width=2pt] plot [smooth] coordinates {(-22.5:3) (-33.75:2.2) (-45:3)};

\draw[blue, line width=2pt] plot [smooth] coordinates {(-67.5:3) (-90:2.2) (-112.5:3)};

\draw[blue, line width=2pt] (45:3)--(45:4.2);
\draw[red, line width=2pt] (3,3)--(4.5,3);
\draw[red, line width=2pt] (4.5,3)--(4.5,4.5);
\draw[red, line width=2pt] (4.5,4.5)--(3,4.5);
\draw[red, line width=2pt] (3,3)--(3,4.5);
\centrearc[red, line width=1pt, ->](3.75,3.75)(135:45:.4)

\foreach \numpt in {0,1,2,3,4,-1,-2,-3,-4,-5}{\fill (22.5*\numpt:3) circle(4pt);}

\foreach \numpt in {2}{\fill (22.5*\numpt:4.2) circle(4pt);}

\centrearc[red, line width=2pt, ->](0,0)(135:45:1)

\node[below left] at (0.5,2.9) {$m-1$};
\node[below] at (1,2.7) {$m$};
\node[below left] at (2.2,2.2) {$1$};
\node[right] at (2.8,1.3) {$2$};
\node[right] at (3,0) {$3$};
\node[right] at (2.8,-1.3) {$4$};
\node[right] at (2,-2.5) {$5$};
\node[right] at (1,-3.2) {$6$};
\node[below] at (0,-3) {$7$};
\node[left] at (-1,-3.2) {$8$};

\node[right] at (4.5,3) {$\gamma$};
\node[right] at (4.5,4.5) {$\beta$};
\node[above] at (3,4.5) {$\alpha$};
\fill (4.5,3) circle(4pt);
\fill (4.5,4.5) circle(4pt);
\fill (3,4.5) circle(4pt);

\node[above left] at (3,3) {$1 '$};

\end{tikzpicture}
\hfill{}
\begin{tikzpicture}[scale=0.6]

\foreach \numpt in {0,1,2,3,4,-1,-2,-3,-4,-5,-6}{\draw[red, line width=2pt] (22.5*\numpt:3)--(22.5*\numpt+22.5:3);}

\centrearc[red, line width=2pt, dashed](0,0)(112.5:225:3)

\draw[blue, line width=2pt] plot [smooth] coordinates {(0:3) (11.25:2.2) (22.5:3)};

\draw[blue, line width=2pt] plot [smooth] coordinates {(-22.5:3) (-33.75:2.2) (-45:3)};

\draw[blue, line width=2pt] plot [smooth] coordinates {(-67.5:3) (-90:2.2) (-112.5:3)};

\draw[blue, line width=2pt] (90:3)--(90:4.2);
\draw[red, dotted, line width=2pt] (90:4.2)--(90:5.4);
\draw[blue, line width=2pt] (67.5:3)--(67.5:4.2);
\draw[red, dotted, line width=2pt] (67.5:4.2)--(67.5:5.4);
\draw[blue, line width=2pt] (45:3)--(45:4.2);
\draw[red, line width=2pt] (3,3)--(4,3);
\draw[blue, line width=2pt] (4,3)--(4,2);
\draw[red, line width=2pt] (4,2)--(4,1);
\draw[blue, line width=2pt] (4,1)--(4,0);
\draw[red, line width=2pt] (4,0)--(4,-1);
\draw[black, line width=1pt, dashed,] (4,-1)--(4,-1.5);
\draw[blue, line width=2pt] (4,-1.5)--(4,-2);
\draw[red, line width=2pt] (4,-2)--(4,-2.5);
\draw[blue, line width=2pt] (4,-2.5)--(4,-3);
\draw[red, line width=2pt] (4,-3)--(4,-3.5);
\foreach \numpt in {-0.5,3,4,5,6}{\fill  (4,\numpt-3) circle(4pt);}

\foreach \numpt in {0,1,2,3,4,-1,-2,-3,-4,-5}{\fill (22.5*\numpt:3) circle(4pt);}

\foreach \numpt in {-4,-5,-5.5,-6}{\fill (4,\numpt+3) circle(3pt);}

\foreach \numpt in {2,3,4}{\fill (22.5*\numpt:4.2) circle(4pt);}
\foreach \numpt in {3,4}{\fill (22.5*\numpt:5.4) circle(4pt);}

\centrearc[red, line width=2pt, ->](0,0)(135:45:1)

\node[below left] at (0.5,2.9) {$m-1$};
\node[below] at (1,2.7) {$m$};
\node[below left] at (2.2,2.2) {$1$};
\node[right] at (2.8,1.3) {$2$};
\node[right] at (3,0) {$3$};
\node[right] at (2.8,-1.3) {$4$};
\node[right] at (2,-2.5) {$5$};
\node[right] at (1,-3.2) {$6$};
\node[below] at (0,-3) {$7$};
\node[left] at (-1,-3.2) {$8$};

\node[right] at (4,3) {$\alpha_1$};
\node[right] at (4,2) {$\alpha_1 '$};
\node[right] at (4,1) {$\beta_1$};
\node[right] at (4,0) {$\beta_1 '$};
\node[right] at (4,-1) {$\alpha_2$};

\node[above] at (3,3) {$1 '$};
\node[right] at (1.7,4) {$m '$};
\node[left] at (0,4.1) {$m -1'$};
\node[left] at (0,5.2) {$(\gamma)$};
\node[right] at (2.1,5) {$(\delta)$};

\node[right] at (4,-3.5) {$\alpha_{a+1}$};

\end{tikzpicture}
\caption{Construction \ref{con:nevenmgeq10} when $m=n \equiv 0$ mod $4$ and otherwise.}\label{fig:nevenmgeq10}
\end{centering}
\end{figure}

\begin{prop}\label{prop:nevenmgeq10}
Let $m,n$ be even such that $10 \leq m \leq n$ and let $n$, $r$ and $t$ be defined by Construction \ref{con:nevenmgeq10}. Then $\mathcal{M}(G:r,s)$ is a chiral map of type $\{ m,n \}$ and $G$ is an alternating group.
\end{prop}

\begin{proof}
By its definition the map has the correct type $\{m, n \}$.
To demonstrate that the group is primitive, again we may employ earlier arguments: there is just one point ($2$) fixed by both $rt$ and also by its conjugate $(rt)^{r^2}$, and meanwhile the other orbits of $rt$ (which is the product of a long cycle with a transposition and two fixed points) are fused by the action of $(rt)^{r^2}$.

When $m=n \equiv 0$ modulo $4$, $r^4t$ is a single cycle fixing three points:

\noindent $(1,4,6,10, \dots ,  m-2, 3, 7, \dots, m-1, 2, 8, \dots, m, 5, 9, \dots , m-3, 1')(\alpha)(\beta)(\gamma)$.
In this case the group $G$ is alternating by Theorem \ref{thm:JonesAltSym}.

Otherwise when $ m \equiv 0, 2$ mod $4$, the permutation $r^4t$ is respectively

\noindent $(1,4,6,10, \dots ,  m-2, 3, 7, \dots, m-1', m-1, 2, 8, \dots, m', m, 5, 9, \dots , m-3, 1')t_a $
$(1,4,6,10, \dots ,  m',m,5,9, \dots, m-1',m-1,2,8, \dots, m-2, 3,7, \dots , m-3, 1')t_a .$

\noindent In these cases the permutation $(r^4t)^2$ fixes every point labelled with a Greek letter and has cyclic decomposition consisting of precisely one cycle of length $m+3$. When $a \geq 1$ or $i=2$, this permutation has enough fixed points to be able to call on Jones' version of Jordan's Theorem \ref{thm:JonesAltSym} to prove the alternating claim. When $a=0$ and $i=0$, then $m=n$, and the only remaining case is when $m \equiv 2$ mod $4$. Consider the permutation $r^2t = (1,2,5,7, \dots,m-1',m-1,1')(3,4,8,\dots, m',m)(6)(\alpha_1)$ which, when $m \equiv 2$ mod $4$, has cycles of coprime odd lengths $\frac{1}{2}m +2$ and $\frac{1}{2}m$ so the permutation $(r^2t)^{\frac{1}{2}m}$ is a single cycle. Applying Theorem \ref{thm:JonesAltSym} for one final time, we may conclude that in every case the group $G$ is alternating.

Chirality follows by Lemma \ref{lem:chiraldiag} with $b=2$ and $c=3$ for every $m \geq 12$, while in the special case where $m=10$ we use Lemma \ref{lem:chiraldiagsfixedpoint} with the same parameters.
\end{proof}

\section{Proof of Theorem \ref{thm:Thispaper}}\label{sec:proof}

In this section we combine all the results from the previous sections to prove our theorem, which we reproduce here for ease of reference:
\emph{Given a hyperbolic type $\{m ,n \}$, there exists a chiral map of that type with alternating automorphism group $A_k$, for some degree $k$.}

\begin{proof}
When $m$ and $n$ are both odd we may rely on the work of Conder, Huc\'{i}kov\'{a}, Nedela and \v{S}ir\'{a}\v{n} as follows. The orientation-preserving automorphism group $G = \langle r,s \rangle$, from Theorem \ref{thm:CHNS} is either alternating or symmetric and can be generated by two elements, one of order $m$ and the other of order $n$, whose product is an involution. Since $m$ and $n$ are both odd, these are even permutations and the claim is immediate.

For the cases when precisely one of $m$ or $n$ is odd, we assume $m$ is even and we then split the situation according to the relative sizes of $m$ and $n$.

When even $m < n$ odd we may use the constructions in this paper combined with Propositions \ref{prop:noddmis4}, \ref{prop:noddmis6}, \ref{prop:noddmis8}, and \ref{prop:noddmgeq10} to cover all but finitely many cases. In particular, when $m < n$:
Proposition \ref{prop:noddmis4} addresses when $m=4$ and $n \geq 9$, leaving the hyperbolic types $\{ 4,5 \}$ and $\{ 4,7 \}$ (and dual types  $\{ 5,4 \}$ and $\{ 7,4 \}$) still under question;
Proposition \ref{prop:noddmis6} proves the claim when $m=6$ and $n \geq 9$ and so leaves the existence of an alternating chiral map of type $\{ 6,7 \}$ (and so also the type $\{ 7,6 \}$) yet to be determined;
Proposition \ref{prop:noddmis8} covers the claim when $m=8$ and $n \geq 9$;
and Proposition \ref{prop:noddmgeq10} completes this part by proving the claim for all types $\{m, n \}$ where $10 \leq m < n$.

Next we address the cases where odd $n < m$ which is even.
In the case where $n=3$ we use Theorem \ref{thm:BCC} from Bujalance, Conder and Costa, noting that the only missing values for $m$ are for non-hyperbolic types.
Proposition \ref{prop:mgeqnequals5} demonstrates that when $n=5$ the claim is true for all even $m \geq 14$.
When $n =7$ and even $m \geq 16$ the question of existence of a chiral map of type $\{m, n \}$ with alternating group is dealt with by Proposition \ref{prop:mgeqnequals7or9} as is the case when $n = 9$ and $m \geq 22$.
Proposition \ref{prop:mgeqngeq7} shows that when odd $n \geq 7$ and even $m$ is such that $ n+1 \leq m \leq 2n-6$ there is a chiral map of type $\{m, n \}$ with alternating automorphism group while Proposition \ref{prop:mgeqngeq11} proves the same claim for $n \geq 11$ whenever even $m$ is such that $m \geq 2n-6$.

When both $m$ and $n$ are even, we call on duality and use: Proposition \ref{prop:nevenmis4} for $m=4$ and $n \geq 10$; Proposition \ref{prop:nevenmis6} for $m=6$ and $n \geq 10$; Proposition \ref{prop:nevenmis8} for $m=8$; and Proposition \ref{prop:nevenmgeq10} for $m \geq 10$. Small modifications to the constructions yield examples for types $\{ 4,6 \}$, $\{ 4,8 \}$ and $\{ 6,8 \}$ which are given in section \ref{sec:neven}. The only such type still missing is $\{6,6 \}$.

Assuming $m$ is even, and working up to duality, an example pair of generators for a chiral map with alternating group is given in Table \ref{tab:missing} for each of the missing types. 

\begin{table}
\begin{tiny}
\begin{tabular}{|c|c|c|c|}
\hline
Type & $s$ & $t$ & $G$ \\
\hline
$\{4,5\}$ & $(1,2,3,4,5)$ & $(1,3)(4,6)$ & $A_6$ \\

$\{6, 5 \}$ & $(1,2,3,4,5)(6,7,8,9,10)$ & $(1,6)(2,4)(7,8)(9,10)$ & $A_{10}$ \\

$\{8, 5 \}$ & $(1,2,3,4,5)(6,7,8,9,10)$ & $(1,6)(2,4)$ & $A_{10}$ \\

$\{10, 5 \}$ & $(1,2,3,4,5)(6,7,8,9,10)$ & $(1,6)(2,4)(7,11)(8,12)$ & $A_{12}$ \\

$\{12, 5 \}$ & $(1,2,3,4,5)(6,7,8,9,10)$ & $(1,6)(2,4)(7,11)(8,12)(9,13)(10,14)$ & $A_{14}$ \\

& & & \\

$\{ 6,6 \}$ & $(1,2,3,4,5,6)(7,8,9,10,11,12)$  & $(1,7)(2,5)(6,13)(8,12)(9,14)(10,15)$  & $A_{15}$  \\

& & & \\

$\{ 4,7 \}$ & $(1,2,3,4,5,6,7)$ & $(1,5)(2,3)(4,9)(7,8)$ & $A_9$\\

$\{ 6,7 \}$ & $ (1,2,3,4,5,6,7)(8,9,10,11,12,13,14)$ & $ (1,3)(4,6)(7,8)(9,13)(10,15)(11,16)$ & $A_{16}$ \\

$\{10,7\}$ & $(1,2,3,4,5,6,7)(8,9,10,11,12,13,14)$  & $(1,8)(2,4)(5,6)(9,10)$  & $A_{14}$  \\

$\{12,7\}$ &  $(1,2,3,4,5,6,7)(8,9,10,11,12,13,14)$ & $(1,8)(2,3)(4,5)(6,15)(7,16)(9,17)(10,12)(13,14)$  & $A_{17}$  \\
 
$\{14,7\}$ & $(1,2,3,4,5,6,7)(8,9,10,11,12,13,14)$  &$(1,8)(7,9)(10,15)(11,16)$   & $A_{16}$  \\
 
& & & \\

$\{14,9\}$ & $(1,2,3,4,5,6,7,8,9)(10,11,12) := S_9$  &$(1,10)(2,12)(3,13)(4,14)(5,15)(6,16)$   & $A_{16}$  \\

$\{16,9\}$ & $S_9(13,17,18)$  & $(1,10)(2,12)(3,13)(4,14)(5,15)(6,16)$  & $A_{18}$  \\

$\{18,9\}$ & $S_9(13,17,18)(14,19,20)$  & $(1,10)(2,12)(3,13)(4,14)(5,15)(6,16)$  & $A_{20}$  \\

$\{20,9\}$ & $S_9(13,17,18)(14,19,20)(15,21,22)$  & $(1,10)(2,12)(3,13)(4,14)(5,15)(6,16)$  & $A_{22}$  \\
\hline
\end{tabular}\caption{Filling gaps - the missing types}\label{tab:missing}
\end{tiny}
\end{table}

The theorem has been proven for all hyperbolic types $\{ m,n \}$ where $m$ is even, and so by the well-known fact that the dual of a chiral map is chiral with the same automorphism group, the proof is complete.
\end{proof}

\section*{Acknowldgement}
The author is very grateful to Jozef \v{S}ir\'{a}\v{n} for asking the question which this paper answers, for his generosity with time and advice, for reading an earlier draft and for many enjoyable related discussions.

\end{document}